\documentclass[onefignum,onetabnum]{siamart251216}

\usepackage{braket,amsfonts}
\usepackage{amssymb} 
\usepackage{amsmath}
\usepackage[numbers]{natbib}
\usepackage{array}

\usepackage{soul}

\usepackage[caption=false]{subfig}
\usepackage{pgfplots}

\newsiamthm{claim}{Claim}
\newsiamremark{remark}{Remark}
\newsiamremark{hypothesis}{Hypothesis}
\crefname{hypothesis}{Hypothesis}{Hypotheses}

\usepackage{algorithmic}
\usepackage{algorithm}
\usepackage{graphicx,epstopdf}

\Crefname{ALC@unique}{Line}{Lines}

\usepackage{amsopn}

\usepackage[normalem]{ulem}

\usepackage{xspace}
\usepackage{bold-extra}
\usepackage[most]{tcolorbox}

\colorlet{texcscolor}{blue!50!black}
\colorlet{texemcolor}{red!70!black}
\colorlet{texpreamble}{red!70!black}
\colorlet{codebackground}{black!25!white!25}

\newsiamthm{assumption}{Assumption} 

\lstdefinestyle{siamlatex}{%
  style=tcblatex,
  texcsstyle=*\color{texcscolor},
  texcsstyle=[2]\color{texemcolor},
  keywordstyle=[2]\color{texemcolor},
  moretexcs={cref,Cref,maketitle,mathcal,text,headers,email,url},
}

\tcbset{%
  colframe=black!75!white!75,
  coltitle=white,
  colback=codebackground, 
  colbacklower=white, 
  fonttitle=\bfseries,
  arc=0pt,outer arc=0pt,
  top=1pt,bottom=1pt,left=1mm,right=1mm,middle=1mm,boxsep=1mm,
  leftrule=0.3mm,rightrule=0.3mm,toprule=0.3mm,bottomrule=0.3mm,
  listing options={style=siamlatex}
}

\newtcblisting[use counter=example]{example}[2][]{%
  title={Example~\thetcbcounter: #2},#1}

\newtcbinputlisting[use counter=example]{\examplefile}[3][]{%
  title={Example~\thetcbcounter: #2},listing file={#3},#1}

\DeclareTotalTCBox{\code}{ v O{} }
{ 
  fontupper=\ttfamily\color{black},
  nobeforeafter,
  tcbox raise base,
  colback=codebackground,colframe=white,
  top=0pt,bottom=0pt,left=0mm,right=0mm,
  leftrule=0pt,rightrule=0pt,toprule=0mm,bottomrule=0mm,
  boxsep=0.5mm,
  #2}{#1}

\patchcmd\newpage{\vfil}{}{}{}
\begin{tcbverbatimwrite}{tmp_\jobname_header.tex}
\title{Residuals-Based Contextual Distributionally Robust Optimization with Decision-Dependent Uncertainty: Theoretical Guarantees and Decomposition Algorithm}

\author{Qing Zhu\thanks{Department of Integrated Systems Engineering, The Ohio State University, Columbus, OH, USA, (\email{zhu.2166@osu.edu}).}
\and Xian Yu\thanks{Corresponding author; Department of Integrated Systems Engineering, The Ohio State University, Columbus, OH, USA,  (\email{yu.3610@osu.edu}).}
\and G\"uz\.{I}n Bayraksan\thanks{Department of Integrated Systems Engineering, The Ohio State University, Columbus, OH, USA, (\email{bayraksan.1@osu.edu}). This author is partially supported by the U.S.\ Department of Energy, Office of Science, Office of Advanced Scientific Computing Research (ASCR) under Grant DE-SC0023361.}}

\headers{Residuals-based Contextual Decision-Dependent DRO}{Qing Zhu, Xian Yu, and G\"uz\.{I}n Bayraksan}
\end{tcbverbatimwrite}
\input{tmp_\jobname_header.tex}

\begin{document}
\maketitle

\begin{tcbverbatimwrite}{tmp_\jobname_abstract.tex}
\begin{abstract}
We consider a residuals-based distributionally robust optimization (DRO) model, where the underlying uncertainty depends on both covariate information and our decisions. 
We adopt both parametric and nonparametric regression models to learn the latent decision dependency and construct a nominal distribution (thereby ambiguity sets) around the learned model using empirical residuals from the regressions. 
We formulate the ambiguity set via the Wasserstein distance, where the nominal distribution is both decision- and covariate-dependent.  
We provide conditions under which desired statistical properties such as asymptotic optimality, rate of convergence, and finite sample guarantees are satisfied. To solve the resulting DRO model, we develop a specialized Bender's decomposition algorithm with nonlinear cuts and prove its finite convergence. Through numerical experiments, we illustrate the effectiveness of our approach and the benefits of integrating decision dependency into a residuals-based DRO framework.
\end{abstract}



\end{tcbverbatimwrite}
\input{tmp_\jobname_abstract.tex}

\section{Introduction}

Many real-world optimization problems under uncertainty face the following two key complicating factors: (i) the underlying uncertainty is often affected by contextual/covariate information, and (ii) the decisions to be optimized can also have a significant impact on the uncertainty.
For example, in facility location problems, the decision maker needs to decide where to open new stores to sell a product in order to maximize total revenue under uncertain customer demand. Customer demand could be affected by contextual information (e.g., seasonality, market indicators) as well as facility location decisions.
Opening a facility in an area could increase the demand in that area, leading to the so-called \textit{endogenous} or \textit{decision-dependent} uncertainty. 
Another example in the realm of power systems is that accurately predicting electricity demand entails utilizing contextual factors (e.g., seasonal variations), while decisions regarding electricity generation, expansion, and distribution can themselves influence the electricity demand.  Therefore, in many real-world problems, it is essential to capture the impact of both contextual information and decision dependency on the underlying uncertainty.

To leverage covariate information in decision-making problems under uncertainty, the following contextual (or conditional) stochastic program (CSP) has been proposed \cite{ban2019big,bertsimas2020predictive,sen2018learning}:
\begin{equation}
    \label{eq:csp}
     \min_{z \in \mathcal{Z}} \mathbb{E}_Y[c(z,Y)|X=x],
     \tag{\bf{CSP}} \notag
\end{equation}
where $z$ denotes the decision variables with feasible set $\mathcal{Z} \subseteq \mathbb{R}^{d_z}$, $X \in \mathbb{R}^{d_x}$ denotes the random vector of covariates with $x$ being its realization, and the random vector $Y \in \mathbb{R}^{d_y}$ denotes the uncertain parameters in the model. In the above setup, the decision maker typically has access to joint observations of $(X,Y)$, and the covariate $x$ is observed before the optimization model is solved. 
In recent years, there has been a stream of research that focuses on solving the above \eqref{eq:csp}; see, e.g., the recent survey \cite{sadana2024survey}. We will briefly review these works in Section \ref{sec: related work}. Among them, \cite{ban2019dynamic} and \cite{sen2018learning} proposed adding residuals from the prediction model 
to account for estimation error in the prediction step. 
\cite{kannan2025data} formalized this approach and investigated the theoretical properties of the so-called empirical residuals-based sample average approximation (ER-SAA), and  \cite{kannan2020residuals}  proposed distributionally robust optimization (DRO) variants (denoted ER-DRO). However,  none of these works considered decision dependency on the uncertainty, which is critical in many real-world applications. 
This work extends the residuals-based approach to CSP by modeling decision-dependent uncertainty.

Motivated by this gap and the decision-dependent uncertainty present in many real-world applications, 
we focus on the following {\it decision-dependent CSP}, where the uncertain parameter $Y$ depends on both covariate $x$ and our decisions $z$: 
\begin{equation}\label{eq:ddcsp}
   \min_{z \in \mathcal{Z}} \mathbb{E}_Y [c(z, Y)|X = x, Z =z].
   \tag{\bf{DD-CSP}}
\end{equation}
Above, decisions $z$ can be regarded as another covariate that can impact the uncertainty $Y$, similar to $x$. For notational simplicity, we omit $x$ and $z$ from the notation for $Y$  here. We will write $Y$ as $Y(x,z)$ explicitly to facilitate our analysis later.
Note that the above setup, without loss of generality, considers a portion of decisions $z$ that affect uncertainty and others that may not. 

Because the conditional distribution of $Y$ given the covariates $x$ and decisions $z$ is typically unknown (but assumed to exist) and the resulting conditional expectation (assumed to be well defined and finite) often cannot be calculated, \eqref{eq:ddcsp} cannot be solved exactly. In practice, given joint random observations of $(X, Z, Y)$, drawn from an unknown ground truth distribution, 
the problem \eqref{eq:ddcsp} is instead approximated. 
This approximation is often challenging due to data limitations. For instance, even in large datasets, conditioning on covariates $x$ and decisions $z$ results in only a few effective observations, leading to significant ambiguities in the underlying conditional distribution.

Ambiguities in estimating the conditional distribution in the decision-dependent setting motivate us to use a DRO framework because DRO can have better out-of-sample performance by reducing the so-called ``optimizer's curse'' or ``negative bias'' of stochastic optimization. 
Therefore, to approximate the above \eqref{eq:ddcsp}, we propose an empirical residuals-based decision-dependent DRO (ER-$\rm{D^3RO}$) approach. We first apply a regression model to estimate the dependency of random parameter $Y$ on the covariate $x$ and decision $z$. The residuals obtained during the training step of the regression model are then added to this point prediction to construct an empirical distribution. Centered at this empirical distribution, we construct 
Wasserstein distance based 
ambiguity sets to find optimal decisions against the worst-case scenario within the distributional ambiguity.
Before summarizing our contributions, let us review related research.

\subsection{Related Work}\label{sec: related work}
This work contributes to three lines of research: (i) DRO, (ii) contextual stochastic optimization, and (iii) decision-dependent stochastic optimization. For DRO, we refer the readers to the extensive surveys \cite{rahimian2019distributionally,kuhn2025distributionally}.  We instead focus on contextual and decision-dependent stochastic optimization and look into the intersections between these three research areas.  

We begin with contextual stochastic optimization.
While traditional methods consider purely statistical error in the prediction step without considering the downstream optimization, there has been a series of works that integrated optimization and prediction. We refer interested readers to~\cite{qi2022integrating, sadana2024survey} for reviews on integrating prediction and optimization under a contextual setting.
In one such integrated approach, \cite{elmachtoub2022smart} proposed a new smart ``predict, then optimize'' (SPO) framework with an SPO loss function that measures the decision error induced by a prediction and demonstrated its consistency and asymptotic validity; see, e.g., \cite{el2019generalization,estes2023smart,kallus2023stochastic} for extensions of this framework. 
Another approach, ``estimate-then-optimize'' (ETO),  incorporates learning of the conditional distributions into the optimization step.  
For example, \cite{bertsimas2020predictive} investigated nonparametric regression models by assigning weights to each data point that depend on the covariates, creating a reweighted SAA, and \cite{bertsimas2019dynamic} extended this approach to multistage problems. 
The residuals-based approach, in contrast, creates an SAA approximation by using both the point prediction at the new covariate and the empirical residuals obtained from the learning step, where both parametric and nonparametric methods can be used \cite{ban2019dynamic,sen2018learning,kannan2025data,kannan2020residuals}. 
Besides SPO- and ETO-type models, there is another stream of research that seeks to learn the optimal feature-to-decision mapping, e.g., by decision rules and other methods \cite[e.g.,][]{bertsimas2022data,ban2019big,oroojlooyjadid2020applying,qi2020learning,zhang2024optimal}.

Let us now briefly discuss decision-dependent (or endogenous) uncertainty in stochastic optimization. 
Modeling uncertainty with decision dependency is critical in many applications because failing to take this into account can lead to severely suboptimal solutions. This topic, therefore, has been extensively studied 
in 
 dynamic programming~\citep[e.g.,][]{webster2012approximate}, stochastic programming \citep[e.g.,][]{goel2006class,hellemoetal18, lee2012newsvendor}, and robust optimization \citep[e.g.,][]{nohadani2018optimization, poss2013robust}. 
There are two types of decision dependency. In the first one, decisions affect probability distributions, and in the second one decisions affect when the uncertainty is revealed. Historically, the second type received more attention in the stochastic programming literature \citep[e.g.,][]{goel2006class}, with recent works investigating decision-dependent information discovery \citep{vayanos2011decision,vayanos2026robust}. 
Our work considers the first type. 

We now turn to the pairwise intersections of these three topics. 
Recently, there has been a vein of models that integrate decision dependency in a DRO framework \cite[e.g.,][]{luo2020distributionally,basciftci2021distributionally,noyan2018distributionally,yu2022multistage}. 
Similarly, 
there has been a number of works in the intersection of DRO and contextual stochastic optimization \cite[e.g.,][]{bertsimas2019dynamic,kannan2020residuals,bertsimas2017bootstrap,hanasusanto2013robust}. 
However, the literature on contextual stochastic optimization with decision-dependent uncertainty is scarce.
To date, only a very few works have considered this important setup.  
\cite{bertsimas2020predictive,bertsimas2022data} studied extensions of their methods when decisions affect the conditional distributions/expectations; see also \citep{Liu2023solving} for specialization of \citep{bertsimas2020predictive} 
to a pricing newsvendor problem. 
 On the other hand, \cite{liu2022coupled} used local linear regression models to predict decision-dependent uncertainty in problems without additional covariates. 
Note that this topic has been identified as an active and important future research area in the survey \cite{sadana2024survey}. 

To our knowledge, there exist no works that lie in the intersection of all three topics, and this is the first study that integrates decision-dependent DRO within a contextual setting. 
Motivated by the important real-world applications with decision-dependent uncertainty and the prevalent limited-data regime in these applications, it builds upon the work of \cite{kannan2025data,kannan2020residuals} by incorporating decision dependency in a residuals-based contextual DRO.  This causes several subtleties and necessitates updates to the theoretical and computational methods, which will be investigated in later sections. 

\vspace*{-0.1in}
\subsection{Contributions}
We summarize the main contributions of this paper as follows. First, we introduce decision-dependent uncertainty within the residuals-based DRO framework for approximating the solution to problem \eqref{eq:ddcsp}. We formulate this model using a Wasserstein distance-based ambiguity set that utilizes both covariate and decision information. Second, we study asymptotic optimality, rate of convergence, and finite sample guarantees using our proposed model. Third, to solve the resulting ER-$\rm{D^3RO}$ model more efficiently, we develop a specialized Bender's decomposition algorithm with nonlinear cuts and establish its finite convergence under mild conditions for a class of ER-$\rm{D^3RO}$. Finally, we conduct numerical experiments on a shipment planning and pricing problem, where pricing decisions affect uncertain demand. We compare the efficacy of our ER-$\rm{D^3RO}$ model with its counterparts that ignore decision-dependent uncertainty or distributional robustness in this problem. Our results illustrate the benefits of incorporating decision-dependency in contextual DRO.  

We end with a remark that our proposed modeling framework accommodates a wide range of learning methods (parametric or nonparametric), including linear regression, Nadaraya-Watson kernel regression, and neural networks, among others. However, because of decision-dependent uncertainty, these learning methods must admit representations that can be embedded into optimization models. 
This requirement distinguishes our setting from its decision-independent counterpart, which does not require such representations inside the resulting optimization model.

\section{Problem Formulation}\label{sec:background}

\subsection{Preliminaries and Setting}\label{sec:preliminary}

Suppose that the ground truth relationship between the random parameter $Y$, the covariate $X$, and the decision $Z$ is given by $Y = f^*(X, Z) + \epsilon$, where $\epsilon$ denotes the associated regression error with zero mean that is independent\footnote{It is possible to consider heteroscedastic cases for which the error may depend on both $X$ and $Z$ by using the model  $Y = f^*(X, Z) + Q(X,Z) \epsilon$, where $Q(X,Z)$ is the square root of the conditional covariance matrix of the error. We refer interested readers to \cite{kannan2025data} for such an extension. For simplicity, we consider homoscedastic case in this paper.} of both $X$ and $Z$. Then, $f^*(x,z):=\mathbb{E}[Y|X=x, Z=z]$ describes the true regression function, which is not restricted to be linear. We assume that the true regression model $f^*$ belongs to a function class $\mathcal{F}$. However, the approximation framework and theoretical results discussed in subsequent sections still work if $f^*\not\in \mathcal{F}$ (i.e., when the model is misspecified) by replacing $f^*$ with the best approximation in $\mathcal{F}$, where the convergence and other results updated with respect to this best approximation; we refer interested readers to Remark 1 in~\cite{kannan2020residuals} for detailed discussions.
The supports of $X,\ Y,\ Z$, and $\epsilon$ are denoted by $\mathcal{X}\subseteq \mathbb{R}^{d_x},\  \mathcal{Y} \subseteq \mathbb{R}^{d_y},\ \mathcal{Z} \subseteq \mathbb{R}^{d_z}$, and $\Xi \subseteq \mathbb{R}^{d_y}$. Let $P_{\epsilon}$ and $P_{Y|X=x,Z=z}$ be the distribution of $\epsilon$ and the conditional distribution of $Y$ given $X=x,\ Z=z$, respectively. We assume that
 $\mathcal{Y}$ is a nonempty closed and convex set---which is required for the orthogonal projection onto $\mathcal{Y}$ to be unique and Lipschitz continuous---and $\mathcal{Z}$ is a nonempty and compact set. Under this setting, \eqref{eq:ddcsp} can be written as
\begin{equation}\label{eq: true problem}
  v^*(x):=\min_{z \in \mathcal{Z}} \Bigl\{g(z,x):=\mathbb{E}_\epsilon[c(z,f^*(x,z) + \epsilon)]\Bigr\},
\end{equation}
where $\mathbb{E}_\epsilon[\cdot]$ denotes expectation taken with respect to the distribution of $\epsilon$. Notice that  problem \eqref{eq: true problem} is well-defined and finite with a nonempty optimal solution set $S^*(x) \subseteq \mathcal{Z}$ if $g(\cdot,x)$ is lower semicontinuous on $\mathcal{Z}$ for a.e.\ $x \in \mathcal{X}$ and $\mathbb{E}_\epsilon[|c(z,f^*(x,z) + \epsilon)|] <+\infty$ for a.e.\ $x \in \mathcal{X},\ z \in \mathcal{Z}$.

Let $D_n:=\{(x^k, z^k, y^k)\}_{k=1}^n$ denote the joint observations of $(X,Z,Y)$. We assume throughout that the decision maker has access to such joint observations. Note that the observations $D_n$ do not have to be independent and identically distributed (i.i.d.). If we know the true regression function $f^*$, we can construct the following full-information decision-dependent SAA (FI-DD-SAA) using the data $D_n$:
\begin{equation}\label{eq:FI-DDSAA}
   \min_{z \in \mathcal{Z}} \Bigl\{g_n^*(z,x):= \frac{1}{n} \sum_{k=1}^n c(z,f^*(x,z) + \epsilon^k)\Bigr\},
   \tag{\bf{FI-DD-SAA}}
\end{equation}
where $\epsilon^k:= y^k - f^*(x^k,z^k)$ is the corresponding error under $f^*$. However, since the true regression function $f^*$ is unknown, we first need to use a regression method on data $D_n$ to get an estimated regression function $\hat{f}_n$. Then we calculate the empirical residuals $\hat{\epsilon}_n^k:= y^k - \hat{f}_n(x^k,z^k)$ from each observation $k$ to construct the following empirical residuals-based decision-dependent SAA (ER-DD-SAA) of problem \eqref{eq: true problem}:
\begin{equation}\label{eq: ersaa}
  \hat{v}_n^{ER}(x):= \min_{z \in \mathcal{Z}} \Bigl\{\hat{g}_n^{ER}(z,x):= \frac{1}{n} \sum_{k=1}^n c\bigl(z,{\rm{proj}}_{\mathcal{Y}}(\hat{f}_n(x,z) +\hat{\epsilon}_n^k)\bigr)\Bigr\},
  \tag{\bf{ER-DD-SAA}}
\end{equation} 
where ${\rm{proj}}_S(y)$ denotes the orthogonal projection of $y$ onto a nonempty closed convex set $S$. 
This projection step helps to avoid any unwanted quantities that can cause infeasibilities. For instance, a negative value of $\hat{f}_n(x,z)+\hat{\epsilon}_n^k$ (e.g., demand) can render the downstream optimization problem infeasible if $c(\cdot,\cdot)$ denotes the optimal value of a second-stage problem. 
In some applications, however, $\mathcal{Y}$ may not be known.
In these cases, a closed convex superset of $\mathcal{Y}$ could be used to leverage partial knowledge (e.g., nonnegativity) about the uncertainty, if such information is known. Alternatively, the superset can be set to $\mathbb{R}^{d_y}$, essentially removing the projection step. The analysis below then follows on this superset.  We continue our analysis with projection onto a known $\mathcal{Y}$. 
We end with a note that, unlike the decision-independent setting where the projection step could be done outside of the optimization, this projection is more complicated in the decision-dependent setting because it also involves the decisions $z$.

Because $\mathcal{Y}$ (or its superset) is nonempty closed and convex, the orthogonal projections are Lipschitz continuous. Then, for all $k=1,\ldots,n$ and for a.e.\ $x \in \mathcal{X},\ z \in \mathcal{Z}$, we have 
\begin{align}\label{eq:projection lip}
    \|{\rm{proj}}_\mathcal{Y} (\hat{f}_n(x,z)+\hat{\epsilon}_n^k)-(f^*(x,z)+\epsilon^k)\| \leq \|\tilde{\epsilon}_n^k(x,z)\|, 
\end{align}
where 
\begin{align}
    \tilde{\epsilon}_n^k(x,z):=&(\hat{f}_n(x,z)+ \hat{\epsilon}_n^k)-(f^*(x,z)+\epsilon^k)\nonumber\\
    =& \underbrace{(\hat{f}_n(x,z)-f^*(x,z))}_{\text{prediction error}} + \underbrace{(f^*(x^k,z^k)-\hat{f}_n(x^k,z^k))}_{\text{estimation error}}. \label{eq:tilde-epsilon}
\end{align}
The first term in \eqref{eq:tilde-epsilon} is the so-called ``prediction error'' at the new covariate $x \in \mathcal{X}$ and decision $z\in\mathcal{Z}$, and the second term in \eqref{eq:tilde-epsilon} is the ``estimation error'' at the training point $(x^k,z^k).$ 
This split is a key component of the subsequent analysis. 
We end this section with a list of common notations used throughout the paper. Additional notation will be introduced as needed in the subsequent sections. 

\paragraph{Notation} Let $\|\cdot\|_q$ denote the $\ell_q$-norm for $q \in [1,+\infty]$, and let $\|\cdot\|$ represent the $\ell_2$-norm as a shorthand. We use $d_x, d_y, d_z$ to denote the dimensions of vectors $X, Y, Z$.
We use the shorthand notation to denote the set $[n]:=\{1,\ldots,n\}$.
Denote ``a.e.'', ``LLN'', ``i.i.d.'', ``w.r.t.'', ``s.t.'', ``Eq.'' and ``r.h.s.'' to be the abbreviations for ``almost every/everywhere'', ``the law of large numbers'', ``independent and identically distributed'', ``with respect to'', ``such that'', ``equation'', and ``right-hand side''. We denote $\xrightarrow{P},\ \xrightarrow{d},\ \xrightarrow{a.s.}$ to be convergence in probability, in distribution, and almost surely, respectively, with respect to the probability measure generating observations of $(X,Z,Y)$. 
Denote by $\delta_y$ the Dirac distribution concentrating unit mass at $y \in \mathbb{R}^{d_y}$. For two sets $S_1,S_2 \subseteq \mathbb{R}^{d_z}$, denote the deviation of $S_1$ from $S_2$ to be $\mathbb{D}(S_1,S_2):= \sup_{a \in S_1}{\rm{dist}}(a, S_2)$, where ${\rm{dist}}(a,S_2):=\inf_{b \in S_2}\|a-b\|.$ Denote $o_p$ and $O_p$ to be convergence in probability to zero and bounded in probability, respectively.

\subsection{Formulation}
\label{sec:formulation}
We now present a data-driven DRO formulation to approximate \eqref{eq: true problem} using the Wasserstein distance to construct ambiguity sets, where the ambiguity set depends on both the covariates $x$ and the decisions $z$. Let us start by introducing the Wasserstein distance.  Let $\mathcal{P}(S)$ be the space of probability distributions supported on $S \subseteq \mathbb{R}^{d_y}$. 
Denote the $p$-Wasserstein distance between probability distributions $P_1\in \mathcal{P}(S)$ and $P_2 \in \mathcal{P}(S)$ as $d_{W,\,p}(P_1,P_2)$, where the set of joint distributions with marginals $P_1$ and $P_2$ is represented by $\Pi(P_1,P_2)$. Then, the $p$-Wasserstein distance for $p \in [1,+\infty)$ is given by 
\[
d_{W,\, p}(P_1,P_2):=
        \left(\; \inf_{\pi \in \Pi(P_1,P_2)} \int_{S^2} \|y_1 - y_2\|^p d\pi(y_1,y_2)\right)^{1/p}.
\]
In the above definition, we use the $\ell_2$-norm as the reference distance (i.e., $\|y_1 - y_2\|$). However, the theoretical results derived in Section \ref{section: Wasserstein} can also be extended to Wasserstein distances defined using any $\ell_q$-norm  as the reference distance for $q \neq 2$. 

To construct a Wasserstein distance-based ambiguity set, let us formalize the empirical conditional distribution on which this ambiguity set is based. Toward this end, denote $P_n^*(x,z)$ as the {\it true} empirical distribution of $Y$ given $X=x,\ Z=z$ corresponding to \eqref{eq:FI-DDSAA}. Similarly, denote $\hat{P}_n^{ER}(x,z)$ as the {\it estimated} empirical distribution corresponding to \eqref{eq: ersaa}. These two distributions are defined as follows:
\begin{align*}
P_n^*(x,z):=\frac{1}{n} \sum_{k=1}^n \delta_{f^*(x,z)+\epsilon^k}\ \ \ \text{and}\ \ \ \hat{P}_n^{ER}(x,z):= \frac{1}{n} \sum_{k=1}^n \delta_{{\rm{proj}}_\mathcal{Y}(\hat{f}_n(x,z)+\hat{\epsilon}_n^k)}.
\end{align*}
Because the true regression function $f^*$ is not known, the observable empirical distribution $\hat{P}_n^{ER}(x,z)$ is used to form the ambiguity set. On the other hand, the true empirical distribution $P_n^*(x,z)$, while unobservable, forms a critical part of the subsequent analysis. 

To approximate problem \eqref{eq: true problem} under distributional ambiguity, we consider the following decision-dependent contextual DRO model, denoted as the ER-$\rm{D^3RO}$ problem: 
\vspace{-0.3cm}
\begin{equation}\label{eq:main problem}
 \hat{v}_n^{D^3RO}(x) := \min_{z \in \mathcal{Z}} \sup_{Q \in \hat{\mathcal{P}}_n(x,z)} \mathbb{E}_{Y \sim Q} [c(z,Y)]. 
 \tag{\textbf{ER-$\mathbf{\rm{D^3}}$RO}}
\end{equation}
The expectation in the above model is taken with respect to the worst-case distribution $Q$ of $Y(x,z)$ from the ambiguity set $\hat{\mathcal{P}}_n(x,z)$, which is constructed using the Wasserstein distance of order $p \in [1,+\infty)$ as \citep{esfahani2015data,gao2023distributionally}
    \begin{equation} \label{eq:wass as}
        \hat{\mathcal{P}}_n(x,z):=\Bigl\{Q \in P(\mathcal{Y}): d_{W,\, p} (Q,\hat{P}_n^{ER}(x,z))\leq \xi_n(\alpha,x)\Bigr\}.
    \end{equation}
This ambiguity set is centered at the empirical residuals-based distribution $\hat{P}_n^{ER}(x,z)$ with a given radius $\xi_n(\alpha,x) \geq 0$. 
We allow the radius $\xi_n(\alpha,x)$ to depend on a certain risk level $\alpha\in(0,1)$. Throughout this paper, we consider decision-independent radii that may depend on covariates $x$ but not on $z$. However, all theoretical results derived in Section \ref{section: Wasserstein} can be extended to the decision-dependent radii case $(\text{i.e., }\xi_n(\alpha,x,z))$.
  
  We will present computationally tractable reformulations of \eqref{eq:main problem} for a class of problems and regression setups in Section~\ref{sec:decomp}. We first examine the theoretical properties of optimal solutions and optimal value of \eqref{eq:main problem} under finite sample sizes and study their behavior asymptotically as more data becomes available. To facilitate this analysis, let $\hat{z}_n^{D^3RO}(x)$ denote an optimal solution to \eqref{eq:main problem} and $\hat{S}_n^{D^3RO}(x)$ denote the set of optimal solutions. We assume that the objective function of \eqref{eq:main problem} is real-valued and lower semicontinuous on $\mathcal{Z}$ for each $x \in \mathcal{X}$, which ensures that the optimal solution set $\hat{S}_n^{D^3RO}(x)$ is nonempty for each $x \in \mathcal{X}$.

\section{Theoretical Guarantees}
\label{section: Wasserstein}

In this section, we establish the asymptotic optimality, rate of convergence, and finite sample guarantees for \eqref{eq:main problem}.

\subsection{Finite Sample Certificate Guarantee}
We begin with the finite sample certificate guarantee, which identifies conditions under which the optimal value of \eqref{eq:main problem} provides an upper bound on the true (but unknown) expected cost of an optimal solution to \eqref{eq:main problem} with a desired probability. To establish this guarantee, we first make the following assumption for the regression estimate $\hat{f}_n$.
\begin{assumption}\label{ass:regression}
    For a.e.\ $x \in \mathcal{X},\ z \in \mathcal{Z}$ and any risk level $\alpha \in (0,1)$, there exist constants $\kappa_{p,n}(\alpha,x) > 0$ and $\kappa_{p,n}(\alpha) > 0$ such that  
    \begin{align}
        &\mathbb{P}\Bigl\{\big\|f^*(x,z)-\hat{f}_n(x,z)\big\|^p > \kappa_{p,n}^p(\alpha,x)\Bigr\} \leq \alpha, \ \ \text{and}\label{eq:condition2-1}
        \\
        &\mathbb{P}\Biggl\{\frac{1}{n}\sum_{k=1}^n \big\|f^*(x^k,z^k)-\hat{f}_n(x^k,z^k)\big\|^p > \kappa_{p,n}^p(\alpha)\Biggr\} \leq \alpha.\label{eq:condition2-2}
    \end{align}
\end{assumption}
\noindent
Note that for simplicity, we use the same notation $\kappa_{p,n}$ in Eqs.~\eqref{eq:condition2-1} and \eqref{eq:condition2-2}, but they represent different values depending on the argument and context. Here, the superscript $p$ denotes the power of the number.  Both constants $\kappa_{p,n}$ are independent of the decisions $z$. Especially,  Eq.~\eqref{eq:condition2-1} is different than the typical way such results may exist in the literature. For instance, for parametric regressions, such constants normally depend on {\it all} covariates ($x$ and $z$); we further discuss this next.  

Assumption~\ref{ass:regression} holds for parametric regression methods such as Ordinary Least Squares (OLS) and Lasso with $p=2$, where $\kappa_{2,n}^2(\alpha,x,z)= C_1\frac{\|x\|^2+\|z\|^2}{n} \log (\frac{1}{\alpha})$ in Eq.~\eqref{eq:condition2-1} and $\kappa_{2,n}^2(\alpha)=C_2\big(\frac{1}{n}\log(\frac{1}{\alpha})\big)$ in Eq.~\eqref{eq:condition2-2} for some constants $C_1,\, C_2>0$ \cite{hsu2012random,rigollet2023high,bunea2007sparsity}. 
Recall that we assume the feasible set $\mathcal{Z}$ to be nonempty and compact. Then, we can derive a decision-independent constant $\kappa_{2,n}^2(\alpha,x)$ for Eq.~\eqref{eq:condition2-1} by identifying an upper bound on $\|z\|^2$ and thus obtaining  $\kappa_{2,n}^2(\alpha,x)= \frac{C_1\|x\|^2+C_3}{n} \log (\frac{1}{\alpha})$ for some constant $C_3>0$. 
Similar bounds hold for $p \neq 2$.
We note that if Assumption \ref{ass:regression} holds for $p=2$, then it holds for any $p=[1,2)$ with the same constant $\kappa_{p,n}(\alpha,x)=\kappa_{2,n}(\alpha,x)$ by Jensen's (or power mean) inequality. This may be especially useful in the case of sub-Gaussian errors, which include zero-mean Gaussian errors \cite{kannan2020residuals}. 
Nonparametric regression methods, such as $k$-nearest neighbor ($k$NN),
typically satisfy Assumption \ref{ass:regression} with $\kappa_{p,n}^p(\alpha,x)=\kappa_{p,n}^p(\alpha) = C_4(\frac{1}{n}\log(\frac{1}{\alpha}))^{C_5/{(d_x + d_z)}}$ for some constants $C_4,\, C_5 >0$ \citep[][Lemma 10]{bertsimas2019predictions}.

We also make the following assumption about Wasserstein  concentration inequality, which has implications for the data $\mathcal{D}_n$. Such an inequality provides bounds on the probability of how much the empirical distribution $P_n^*(x,z)$ deviates---in terms of the Wasserstein distance---from the true conditional distribution $P_{Y|X=x,Z=z}$.

    \begin{assumption}\label{prop: concentration}
    There exist constants $c_1,\ c_2,\ c_3\ge 0$ such that for all $\kappa>0, \ n\in\mathbb{N},\ x\in \mathcal{X},\ z\in\mathcal{Z}$,
    \begin{align}\label{eq:concentration}
    \mathbb{P} \Bigl\{d_{W,\, p}(P_n^*(x,z),P_{Y|X=x,Z=z}) \geq \kappa \Bigr\} \leq  c_1\exp{(-c_2n\kappa^{c_3})}, 
    \end{align}
where the constants $c_1,c_2,c_3$ depend on $\alpha$, $p$, and $d_y$. 
\end{assumption}

Assumption\ \ref{prop: concentration} is satisfied, for instance, when the observations are i.i.d.\ and the error distribution is light-tailed; see, e.g., \citep[][Theorem 2]{fournier2015rate}.  The light-tail assumption can be satisfied by sub-Gaussian distributions, including Gaussian distribution with mean zero, for $p\in[1,2)$. Assumption\ \ref{prop: concentration} can also be satisfied under non-i.i.d.\ data $\mathcal{D}_n$, such as time-series data \cite{dou2019distributionally}.

Next, we set the radius of the Wasserstein ambiguity set \eqref{eq:wass as} to
\begin{equation}\label{eq:radius def}
    \xi_n(\alpha,x) := \kappa_{p,n}^{(1)}(\alpha,x) + \kappa_{p,n}^{(2)}(\alpha)
\end{equation}
at a given risk level $\alpha\in(0,1)$ and covariate information $x$. Here $\kappa_{p,n}^{(1)}(\alpha,x):= \kappa_{p,n}(\frac{\alpha}{4},x) + \kappa_{p,n}(\frac{\alpha}{4})$, where $\kappa_{p,n}(\frac{\alpha}{4},x)$ and $\kappa_{p,n}(\frac{\alpha}{4})$ are defined in Assumption \ref{ass:regression} through Eq.s \eqref{eq:condition2-1} and \eqref{eq:condition2-2}, and $\kappa_{p,n}^{(2)}(\alpha)$ is defined as 
\begin{align}\label{kappa2}
\kappa_{p,n}^{(2)}(\alpha):= \left(\frac{\log(2c_1 \alpha^{-1})}{c_2n}\right)^{1/c_3},  
\end{align}
which is obtained by setting the r.h.s.\ of Eq.~\eqref{eq:concentration} to $\alpha/2$. The first component $\big(\text{i.e., }\kappa_{p,n}^{(1)}(\alpha,x)\big)$ of the radius in \eqref{eq:radius def} is due to the estimation of the regression function, and the second component $\big(\text{i.e., }\kappa_{p,n}^{(2)}(\alpha)\big)$ is due to the estimation of the true conditional distribution via a Wasserstein ball around the empirical distribution. While the second component is always independent of decisions $z$, the first component of the radius, $\kappa_{p,n}^{(1)}(\alpha,x)$, is expected to be larger than a decision-dependent counterpart $\big(\text{i.e., }\kappa_{p,n}^{(1)}(\alpha,x,z)\big)$ because we use a more conservative result in \eqref{eq:condition2-1} by removing the dependence on decisions $z$. 

In the rest of this section, we use $\hat{\mathcal{P}}_n(x,z;\xi_n(\alpha,x))$ to denote the ambiguity set \eqref{eq:wass as} to highlight its dependence on the radius $\xi_n(\alpha,x)$ given in Eq.~\eqref{eq:radius def}. We are now ready to present the finite sample certificate guarantee for the \eqref{eq:main problem} model.

\begin{theorem}[Finite sample certificate guarantee]\label{thm:fsg}
Suppose Assumptions  \ref{ass:regression} and \ref{prop: concentration} hold and $\alpha \in (0,1)$ is a given risk level. Then for a.e.\ $x \in \mathcal{X}$ and $z \in \mathcal{Z}$, the finite sample certificate guarantee $\mathbb{P}\big\{g(\hat{z}_n^{D^3RO}(x),x) \leq \hat{v}_n^{D^3RO}(x)\big\} \geq 1-\alpha$ holds for \eqref{eq:main problem} under the Wasserstein ambiguity set $\hat{\mathcal{P}}_n(x,z;\xi_n(\alpha,x))$ defined in Eq.~\eqref{eq:wass as} with the radius $\xi_n(\alpha,x)$ defined in Eq.~\eqref{eq:radius def}.
\end{theorem}
\begin{proof}
Due to the definition of \eqref{eq:main problem}, showing the finite sample certificate guarantee is equivalent to showing $\mathbb{P} \big\{d_{W,\, p} (\hat{P}_n^{ER}(x,z),P_{Y|X=x,Z=z}) > \xi_n(\alpha,x)\big\} \leq \alpha$ for a.e.\ $x \in \mathcal{X}$ and $z \in \mathcal{Z}$. By the triangle inequality, we have
\allowdisplaybreaks
\small
\begin{equation*}
d_{W,p}\bigl(\hat P_n^{ER}(x,z),P_{Y|X=x,Z=z}\bigr)
\le \underbrace{d_{W,p}\bigl(\hat P_n^{ER}(x,z),P_n^*(x,z)\bigr)}_{(i)}
+ \underbrace{d_{W,p}\bigl(P_n^*(x,z),P_{Y|X=x,Z=z}\bigr)}_{(ii)}.
\end{equation*}
\normalsize
First, based on Eq.\ \eqref{eq:projection lip} and the definition of the $p$-Wasserstein distance, $(i)$ satisfies
\small
\begin{align*}
d_{W,p}\bigl(\hat P_n^{ER}(x,z),P_n^*(x,z)\bigr)
&\le \left(\frac{1}{n} \sum_{k=1}^n \big\|{\rm{proj}}_\mathcal{Y} (\hat{f}_n(x,z)+\hat{\epsilon}_n^k)-(f^*(x,z)+\epsilon^k)\big\|^p\right)^{1/p} \\
& \le \left(\frac{1}{n}\sum_{k=1}^n |\tilde\epsilon_n^k(x,z)|^p\right)^{1/p}. 
\end{align*}
\normalsize
Therefore, by Assumption\ \ref{ass:regression} and the probability inequality 
$\mathbb P\{A+B >a+b\} \le \mathbb P\{A>a\} + \mathbb\{B >b\}$, we have for a.e. $x \in \mathcal X$ and $ z \in \mathcal Z$
\footnotesize
    \begin{align*}
        & \mathbb P \left\{
d_{W,p}\bigl(\hat P_n^{ER}(x,z),P_n^*(x,z)\bigr)
>
\kappa_{p,n}^{(1)}(\alpha,x)
\right\}
    \\
        \le &\ \mathbb{P}\bigg\{\left(\frac{1}{n} \sum_{k=1}^n \|\tilde{\epsilon}_n^k(x,z)\|^p\right)^{1/p} > \kappa_{p,n} (\frac{\alpha}{4},x) + \kappa_{p,n} (\frac{\alpha}{4})\bigg\} 
    \\
        \overset{(a)}{\le} & \ \mathbb{P}\left\{\big\|f^*(x,z)-\hat{f}_n(x,z)\big\| +  \left(\frac{1}{n}\sum_{k=1}^n \big\|f^*(x^k,z^k) - \hat{f}_n(x^k,z^k)\big\|^p\right)^{1/p} > \kappa_{p,n} (\frac{\alpha}{4},x) + \kappa_{p,n} (\frac{\alpha}{4})\right\} 
    \\
         \leq &\  \mathbb{P}\Bigg\{\big\|f^*(x,z)-\hat{f}_n(x,z)\big\| > \kappa_{p,n}(\frac{\alpha}{4},x)\Bigg\} + \mathbb{P}\left\{\frac{1}{n} \sum_{k=1}^n \big\|f^*(x^k,z^k) - \hat{f}_n(x^k,z^k)\big\|^p > \kappa_{p,n}^p(\frac{\alpha}{4})\right\} 
    \\
         \overset{(b)}{\leq} &\ \frac{\alpha}{4} + \frac{\alpha}{4} = \frac{\alpha}{2},
    \end{align*}
\normalsize
where inequality $(a)$ holds due to the definition of $\tilde{\epsilon}_n^k(x,z)$ in Eq.\ \eqref{eq:tilde-epsilon} and inequality $(b)$ is due to Assumption \ref{ass:regression}.
Second, $(ii)$ can be bounded by Assumption \ref{prop: concentration}, i.e.,
\begin{equation*}
\mathbb P \left\{
d_{W,p}\bigl(P_n^*(x,z),P_{Y|X=x,Z=z}\bigr)
>
\kappa_{p,n}^{(2)}(\alpha)
\right\}
\le \frac{\alpha}{2}.
\end{equation*}
Combining the above, we obtain
\footnotesize
\begin{align*}
    & \mathbb P \biggl\{
d_{W,p}\bigl(\hat P_n^{ER}(x,z),P_{Y|X=x,Z=z}\bigr)
>
\kappa_{p,n}^{(1)}(\alpha,x)+\kappa_{p,n}^{(2)}(\alpha)
\biggr\} 
\\
\le\;&
\mathbb P\biggl\{
d_{W,p}\bigl(\hat P_n^{ER}(x,z),P_n^*(x,z)\bigr)
>
\kappa_{p,n}^{(1)}(\alpha,x)
\biggr\}
+
\mathbb P\biggl\{
d_{W,p}\bigl(P_n^*(x,z),P_{Y|X=x,Z=z}\bigr)
>
\kappa_{p,n}^{(2)}(\alpha)
\biggr\}
\\
\le\;& \alpha.
\end{align*}
\normalsize
\end{proof}

\subsection{Asymptotic Results}
Next, we study the asymptotic properties of the optimal value and optimal solutions of \eqref{eq:main problem}, providing conditions under which they achieve consistency and asymptotic optimality w.r.t.\ true problem \eqref{eq: true problem}. We also establish rates of convergence.  We begin this analysis by making the following two assumptions regarding the cost function $c$. 

\begin{assumption}\label{assum:lower-semicontinuous}
    The function $c(\cdot,Y(x,\cdot))$ is lower semi-continuous on $\mathcal{Z}$ for a.e.\ $x \in \mathcal{X}$. 
     Furthermore, $c(z,Y)$ is continuous on the second argument and
   there exists $C_p \geq 0$ such that
$
        |c(z,Y(x,z))| \leq C_p(1+\|Y(x,z)\|^p) \ \ \forall x \in \mathcal{X},\ \forall z \in \mathcal{Z}.
$
\end{assumption}

\begin{assumption}
\label{ass: lipschitz cont}
(a) If we use $p$-Wasserstein distance with $p \in [1,\infty)$, then for each $z \in \mathcal{Z}$, the function $c(z,Y)$ is Lipschitz continuous on the second argument with Lipschitz constant $L_1(z)$; 
(b) If we use $p$-Wasserstein distance with $p \in [2, \infty)$, then for each $z\in\mathcal{Z}$, $\nabla c(z,Y)$ exists and is Lipschitz continuous on the second argument with Lipschitz constant $L_2(z)$,  where  $\mathbb{E}[\nabla\|c(z,Y)\|^2] < +\infty$. In the case of $p \in [2, \infty)$, it is enough if either of the conditions (a) or (b) hold.
\end{assumption}

Assumption \ref{assum:lower-semicontinuous} in particular is different in the current decision-dependent setting because it must consider the effect of $z$ on the uncertain parameters $Y$.
Note that the first part of Assumption \ref{assum:lower-semicontinuous} is satisfied, for instance, if (i) $c(z,Y)$ is lower semi-continuous on the first argument and either one of the following two conditions hold: (ii-a) the regression function $f^*(x,\cdot)$ is continuous on $\mathcal{Z}$ for a.e.\ $x \in \mathcal{X}$ or (ii-b) $f^*(x,\cdot)$ is lower semi-continuous on $\mathcal{Z}$ 
and $c(z,Y)$ is nondecreasing on the second argument. Condition (i) and Assumption \ref{ass: lipschitz cont} can be satisfied by a fairly large class of stochastic optimization problems, including two-stage stochastic linear and mixed-integer programs with continuous recourse; see, e.g., \cite[][Appendix E]{kannan2025data}. Since our uncertain parameter $Y$ depends on the decision $z$, we further need continuity or lower semi-continuity of the ground truth regression function $f^*(x,\cdot)$ (i.e., (ii-a) or (ii-b)) to ensure the lower semi-continuity of the cost function $c(\cdot,Y(x,\cdot))$ on $\mathcal{Z}$. Note that Condition (ii-a) or (ii-b) is satisfied by a broad range of regression functions (e.g., linear regression, exponential regression, and so forth).

For the class of problems whose cost functions $c$ satisfy the above conditions, we can establish convergence properties of optimal values and solutions of \eqref{eq:main problem} with radius $\xi_n(\alpha_n,x)$ under a suitable sequence of risk levels $\{\alpha_n\}$. We present these conditions next.

\begin{assumption}\label{ass:risk}
    The sequence of risk levels $\{\alpha_n\} \subset (0,1)$ satisfies $\sum_{n}\alpha_n<+\infty$, and $\lim_{n \rightarrow \infty}\xi_n(\alpha_n,x)=0$ for a.e.\ $x \in \mathcal{X}$ with the radius defined in \eqref{eq:radius def}. 
\end{assumption}

We are now ready to establish the asymptotic properties regarding the ER-$\rm{D^3RO}$ model. Toward this end, we first introduce two important lemmas.
The first lemma establishes convergence of the distributions in the ambiguity set, and the second provides useful bounds that play a key role in proving this section's main results. 

\begin{lemma}\label{lemma: Q conv}
Suppose Assumptions \ref{ass:regression}, \ref{prop: concentration}, and \ref{ass:risk} hold. Let $\{Q_n(x,z)\}$ be a sequence of distributions with $Q_n(x,z) \in \hat{\mathcal{P}}_n(x,z; \xi_n(\alpha_n,x))$. Then, 
     (i) 
for a.e.\ $x \in \mathcal{X}$ and each $z\in\mathcal{Z}$, 
    for $n$ large enough, we a.s.\ have
    $d_{W,\, p}(P_{Y|X=x,Z=z},Q_n(x,z))
    \leq 2\xi_n(\alpha_n,x)$. 
     Furthermore, 
     (ii)
     $\{Q_n(x,z)\}$ converges a.s.\ under the $p$-Wasserstein distance
     to $P_{Y|X=x,Z=z}$, that is, $
    \mathbb{P}\Big\{\lim_{n \rightarrow \infty} d_{W,\, p} (P_{Y|X=x, Z=z}, Q_n(x,z))=0\Big\}=1.
    $
\end{lemma}
\begin{proof}
    Following the proof of Lemma 3.7 in \cite{esfahani2015data}, by triangle inequality, we obtain
    \small
    \begin{align*}
        d_{W,\, p}(Q_n(x,z), P_{Y|X=x,Z=z}) 
        \leq & d_{W,\, p} (P_{Y|X=x,Z=z}, \hat{P}_n^{ER}(x,z)) + d_{W,\, p} (Q_n(x,z), \hat{P}_n^{ER}(x,z))\\
     \leq & d_{W,\, p} (P_{Y|X=x,Z=z}, \hat{P}_n^{ER}(x,z)) + \xi_n(\alpha_n,x).
    \end{align*}
    \normalsize
    The proof of Theorem \ref{thm:fsg} implies $ \mathbb{P}\bigl\{d_{W,\, p}(P_{Y|X=x,Z=z}, \hat{P}_n^{ER}(x,z)) \leq \xi_n(\alpha_n,x)\bigr\}\geq 1 - \alpha_n$. As a result, $\mathbb{P}\bigl\{d_{W,\, p}(Q_n(x,z), P_{Y|X=x,Z=z}) \leq 2\xi_n(\alpha_n,x)\bigr\} \geq 1-\alpha_n$. By the  Borel-Cantelli lemma \citep[][Theorem 2.18]{kallenberg1997foundations} and Assumption~\ref{ass:risk}, we obtain (i). Because $\lim_{n \rightarrow \infty}\xi_n(\alpha_n,x)=0$ for a.e.\ $x \in \mathcal{X}$, we obtain (ii). 
\end{proof}

\begin{lemma}\label{lemma: v inequality}
    Suppose Assumptions \ref{ass:regression}, \ref{prop: concentration}, and \ref{ass:risk} hold. Then for $n$ large enough, we a.s.\ have 
    \begin{align}\label{eq:lemma equation}
        v^*(x) \leq g(\hat{z}_n^{D^3RO}(x),x) \leq \hat{v}_n^{D^3RO}(x)
    \end{align}
    for a.e.\ $x \in \mathcal{X}$.
    Let $z^*(x)$ be an optimal solution to  problem \eqref{eq: true problem}. If Assumption \ref{ass: lipschitz cont}(a) holds, then we a.s.\ have for a.e.\ $x \in \mathcal{X}$ and $n$ large enough
    \begin{align}\label{eq: 4a}
    \hat{v}_n^{D^3RO}(x) \leq v^*(x) + 2L_1(z^*(x)) \xi_n(\alpha_n,x).
    \end{align}
    On the other hand, if Assumption \ref{ass: lipschitz cont}(b) holds, then we a.s.\ have for a.e.\ $x \in \mathcal{X}$ and $n$ large enough
    \begin{align}\label{eq:4b}
    \hat{v}_n^{D^3RO}(x) \leq v^*(x) + 2\left(\mathbb{E}[\|\nabla c(z^*(x),Y)\|^2]\right)^{\frac{1}{2}} \xi_n(\alpha_n,x) + 4L_2(z^*(x)) \xi_n^2(\alpha_n,x).
    \end{align}
\end{lemma}
\begin{proof}
    By Theorem \ref{thm:fsg}, we have  
    $\mathbb{P}\bigl\{v^*(x) \leq g(\hat{z}_n^{D^3RO}(x),x) \leq \hat{v}_n^{D^3RO}(x)\bigr\} \geq 1- \alpha_n\ \forall n\in\mathbb{N}$ for a.e.\ $x \in \mathcal{X}$. 
    Then, by the  Borel-Cantelli lemma and Assumption~\ref{ass:risk},
    for $n$ large enough, we a.s.\ have $v^*(x) \leq g(\hat{z}_n^{D^3RO}(x),x) \leq \hat{v}_n^{D^3RO}(x)$ for a.e.\ $x \in \mathcal{X}$.

    Suppose Assumption \ref{ass: lipschitz cont}(a) holds for $p \geq 1$, where $c(z,\cdot)$ is Lipschitz continuous with Lipschitz constant $L_1(z)$. Denote $\bar{\mathcal{P}}_{1,n}(x,z; \xi_n(\alpha_n,x))
    \allowbreak :=\{Q \in P(\mathcal{Y}): d_{W,\, 1} (Q,\hat{P}_n^{ER}(x,z))\leq \xi_n(\alpha_n,x)\}$. Since $\hat{\mathcal{P}}_n(x,z; \xi_n(\alpha_n,x)) \subseteq \bar{\mathcal{P}}_{1,n}(x,z; \xi_n(\alpha_n,x))$ for all  $p \in [1,+\infty)$, we a.s.\ have for a.e.\ $x \in \mathcal{X}$ and $n$ large enough, 
    \begin{align*}
        \hat{v}_n^{D^3RO}(x) &\leq \sup_{ Q \in \bar{\mathcal{P}}_{1,n}(x,z^*(x); \xi_n(\alpha_n,x))} \mathbb{E}_{Y \sim Q} [c(z^*(x),Y)] \\
        &\leq g(z^*(x),x) + 2L_1(z^*(x)) \xi_n(\alpha_n,x),
    \end{align*}
    where the first inequality is because $z^*(x)\in\mathcal{Z}$ is a feasible but potentially suboptimal solution to \eqref{eq:main problem} in addition to changing $\hat{\mathcal{P}}_n$ to $\bar{\mathcal{P}}_{1,n}$ and the second inequality is due to Lemma \ref{lemma: Q conv}(i), Assumption \ref{ass: lipschitz cont}(a), and by Kantorovich-Rubinstein Theorem \citep[e.g.,][Theorem 2]{kuhn2019wasserstein}.

    Suppose Assumption \ref{ass: lipschitz cont}(b) holds for $p \geq 2$, where $\nabla c(z,\cdot)$ is Lipschitz continuous with Lipschitz constant $L_2(z)$. Denote $\bar{\mathcal{P}}_{2,n}(x,z; \xi_n(\alpha_n,x)):=\{Q \in P(\mathcal{Y}): d_{W,\, 2} (Q,\hat{P}_n^{ER}(x,z))\leq \xi_n(\alpha_n,x)\}$. Since  $\hat{\mathcal{P}}_n(x,z; \xi_n(\alpha_n,x)) \subseteq \bar{\mathcal{P}}_{2,n}(x,z; \xi_n(\alpha_n,x))$ for all $p\in [2,+\infty)$, we a.s.\ have for a.e.\ $x \in \mathcal{X}$, 
    \small
    \begin{align*}
        \hat{v}_n^{D^3RO}(x) & \leq \sup_{ Q \in \bar{\mathcal{P}}_{2,n}(x,z^*(x); \xi_n(\alpha_n,x))} \mathbb{E}_{Y \sim Q} [c(z^*(x),Y)]\\
        & \leq g(z^*(x),x) + 2\left(\mathbb{E}[\|\nabla c(z^*(x),Y)\|^2]\right)^{\frac{1}{2}} \xi_n(\alpha_n,x) + 4L_2(z^*(x))\xi_n^2(\alpha_n,x),
    \end{align*}
    \normalsize
    where the second inequality is due to Lemma \ref{lemma: Q conv}(i), Assumption \ref{ass: lipschitz cont}(b) and Lemma 2 in~\cite{gao2023finite}.
\end{proof}

Armed with these lemmas, we now establish the consistency and asymptotic optimality of the optimal value and solutions of \eqref{eq:main problem}. 

\begin{theorem}[Consistency and asymptotic optimality] \label{thm:consistency-wass}
Suppose Assumptions \ref{ass:regression}, \ref{prop: concentration}, \ref{assum:lower-semicontinuous}, and \ref{ass:risk} hold. Then for a.e.\ $x \in \mathcal{X}$, the optimal value and solution of\ \eqref{eq:main problem} under Wasserstein ambiguity set $\hat{\mathcal{P}}_n(x,z;\xi_n(\alpha_n,x))$ defined in Eq.~\eqref{eq:wass as} with the radius $\xi_n(\alpha_n,x)$ defined in Eq.~\eqref{eq:radius def} are consistent and asymptotically optimal, that is,
\[
\hat{v}_n^{D^3RO}(x) \xrightarrow{P} v^*(x),\  \ \text{dist}(\hat{z}_n^{D^3RO}(x),S^*(x)) \xrightarrow{P} 0,\ \  g(\hat{z}_n^{D^3RO}(x),x) \xrightarrow{P} v^*(x).
\]
\end{theorem}
\begin{proof}
    By the proof of Theorem \ref{thm:fsg}, we  have $\mathbb{P} \big\{d_{W,\, p}(\hat{P}_n^{ER}(x,z), P_{Y|X=x, Z=z}) > \xi_n(\alpha_n,x)\big\} \leq \alpha_n$ for a.e.\ $x \in \mathcal{X}$ and $z \in \mathcal{Z}$.
    And by Lemma \ref{lemma: Q conv}(ii), for any $Q_n(x,z) \in \hat{\mathcal{P}}_n(x,z,\xi_n(\alpha_n, x))$, we a.s.\ have
    $\lim_{n \rightarrow \infty}d_{W,\, p}(P_{Y|X=x,Z=z}, \allowbreak Q_n(x,z)) = 0\ \text{for a.e.}\
    x \in \mathcal{X}$ and $z \in \mathcal{Z}.$ 
    Because convergence with respect to the Wasserstein distance implies weak convergence, this means $Q_n(x,z)$ converges weakly to $P_{Y|X=x,Z=z}$ for a.e.\ $x\in\mathcal{X}$ and $z\in\mathcal{Z}$ in the space of distributions with finite $p$-th order moments \citep[][Theorem 6.9]{villani2009optimal}. 
    By Eq.~\eqref{eq:lemma equation} in Lemma \ref{lemma: v inequality}, we know for a.e.\ $x\in\mathcal{X}$, $v^*(x) \leq g(\hat{z}_n^{D^3RO}(x),x) \leq \hat{v}_n^{D^3RO}(x)$ for $n$ large enough.
    Therefore, to show $\lim_{n \rightarrow \infty} \hat{v}_n^{D^3RO}(x) = v^*(x) = \lim_{n \rightarrow \infty} g(\hat{z}^{D^3RO}_n(x),x)$ in probability (or a.s.) for a.e.\ $x \in \mathcal{X}$, it suffices to prove $\limsup_{n \rightarrow \infty} \hat{v}_n^{D^3RO}(x) \leq v^*(x)$ a.s.\ for a.e.\ $x \in \mathcal{X}$. 
    
    Let $z^*(x)$ be an optimal solution to the true problem \eqref{eq: true problem} for a.e.\ $x \in \mathcal{X}$.  Given any $\delta>0$, choose $Q_n^*(x,z^*(x)) \in \hat{\mathcal{P}}_n(x,z^*(x);\xi_n(\alpha_n,x))$ such that it satisfies
    \begin{align}\label{eq: choose qn}
        \sup_{Q \in \hat{\mathcal{P}}_n(x,z^*(x);\xi_n(\alpha_n,x))} \mathbb{E}_{Y\sim Q}[c(z^*(x),Y)] \leq \mathbb{E}_{Y \sim Q_n^*(x,z^*(x))}[c(z^*(x),Y)] + \delta.
    \end{align} 
    Then for a.e.\ $ x \in \mathcal{X}$, we a.s.\ have
    \begin{align*}
        \limsup_{n \rightarrow \infty} \hat{v}_n^{D^3RO}(x)  {\leq}&  \limsup_{n \rightarrow \infty} \sup_{Q \in \hat{\mathcal{P}}_n(x,z^*(x);\xi_n(\alpha_n,x))} \mathbb{E}_{Y \sim Q} [c(z^*(x),Y)] 
        \\
         \overset{(a)}{\leq}& \limsup_{n \rightarrow \infty}\mathbb{E}_{Y \sim Q_n^*(x,z^*(x))} [c(z^*(x),Y)] + \delta 
        \\
        \overset{(b)}{=}& g(z^*(x),x) + \delta 
         = v^*(x) + \delta,
    \end{align*}
    where inequality $(a)$ above holds by the choice of $Q_n^*(x,z^*(x))$ in Eq.~\eqref{eq: choose qn}, and inequality $(b)$ holds by using \citep[][Definition 6.8]{villani2009optimal} because $Q^*_n(x,z^*(x))$ converges weakly to $P_{Y|X=x,Z=z}$ and by the second part of Assumption \ref{assum:lower-semicontinuous}. 
    Since $\delta>0$ is arbitrary, we obtain $\limsup_{n \rightarrow \infty} \hat{v}_n^{D^3RO}(x) \leq v^*(x)$ a.s.\ for a.e.\ $x \in \mathcal{X}$.

    Next, we show that any accumulation point of the solution sequence $\{\hat{z}_n^{D^3RO}(x)\}$ is a.s.\ an element of $S^*(x)$. 
    Let $\bar{z}(x)$ be an accumulation point of $\hat{z}_n^{D^3RO}(x)$. For simplicity, assume $\lim_{n\rightarrow \infty} \hat{z}_n^{D^3RO}(x)=\bar{z}(x)$ (otherwise, use a subsequence). Then for a.e.\ $ x \in \mathcal{X}$, we a.s.\ have
    \begin{align}
    v^*(x) \leq g(\bar{z}(x),x) & \overset{(a)}{\le} \mathbb{E}\left[\liminf_{n \rightarrow \infty} c(\hat{z}_n^{D^3RO}(x),Y(x,\hat{z}_n^{D^3RO}(x)))\right] \nonumber 
    \\
    & \overset{(b)}{\le} \liminf_{n \rightarrow \infty} g(\hat{z}_n^{D^3RO}(x),x) \overset{(c)}\leq v^*(x), \nonumber 
    \end{align}
    where $(a)$ holds due to the lower semicontinuity of $c(\cdot,Y(x,\cdot))$ on $\mathcal{Z}$ for each $x \in \mathcal{X}$ in Assumption \ref{assum:lower-semicontinuous}, and $(b)$ holds due to Fatou's lemma under Assumption \ref{assum:lower-semicontinuous}. Now, from Eq.\ \eqref{eq:lemma equation} in Lemma \ref{lemma: v inequality} and the results derived in the first part of the proof, we have for a.e.\ $ x \in \mathcal{X}$, $\liminf_{n \rightarrow \infty} g(\hat{z}_n^{D^3RO}(x),x) \leq \lim_{n \rightarrow \infty}\hat{v}_n^{D^3RO}(x) = v^*(x)$, which shows inequality (c).    
    Therefore, we a.s.\ have $\bar{z}(x) \in S^*(x)$. 

    The rest of the proof follows directly from the proof of Theorem 9 in \cite{kannan2020residuals}. 
\end{proof}

In order to establish the rate of convergence, as opposed to convergence as in Theorem \ref{thm:consistency-wass}, we additionally need Assumption \ref{ass: lipschitz cont}. Assumption \ref{assum:lower-semicontinuous} on its own was enough for Theorem \ref{thm:consistency-wass}. 

\begin{theorem}[Rate of convergence]\label{thm:roc}
 Suppose Assumptions \ref{ass:regression}--\ref{ass:risk} hold. Then for a.e.\ $x \in \mathcal{X}$, \eqref{eq:main problem} under Wasserstein ambiguity set $\hat{\mathcal{P}}_n(x,z;\xi_n(\alpha_n,x))$ defined in Eq.~\eqref{eq:wass as} with the radius $\xi_n(\alpha_n,x)$ defined in Eq.~\eqref{eq:radius def} satisfies
 \[
 |\hat{v}_n^{D^3RO}(x)-v^*(x)|=O_p(\xi_n(\alpha_n,x))\ \text{and}\  |g(\hat{z}_n^{D^3RO}(x),x)-v^*(x)|=O_p(\xi_n(\alpha_n,x)).
 \]
\end{theorem}
\begin{proof}
    When Assumptions \ref{ass:regression}, \ref{prop: concentration}, \ref{assum:lower-semicontinuous}, and \ref{ass:risk} hold, we have the consistency and asymptotic results in Theorem \ref{thm:consistency-wass}. Lemma \ref{lemma: v inequality} also holds under Assumptions\ \ref{ass:regression}, \ref{prop: concentration}, and~\ref{ass:risk}. Suppose Assumption \ref{ass: lipschitz cont}(a) holds with $p \geq 1$, then the desired result follows from Eq.~\eqref{eq: 4a}. On the other hand, suppose Assumption \ref{ass: lipschitz cont}(b) holds with $p \geq 2$, then the desired result follows from Eq.~\eqref{eq:4b}.
\end{proof}

\subsection{Finite Sample Solution Guarantee} 

We end the theoretical analysis by showing a finite sample result on the optimal solutions $\hat{z}_n^{D^3RO}(x)$ of \eqref{eq:main problem}. 
To establish the finite sample solution guarantee, we consider the following large deviation properties of the regression estimate $\hat{f}_n$, which provides a stronger version of Assumption \ref{ass:regression}.

\begin{assumption}\label{ass:reg constant}
For any given constant $\kappa > 0$, there exist positive constants $K_{p}(\kappa,x), \bar{K}_{p}(\kappa), \beta_{p}(\kappa,x)$, and $\bar{\beta}_{p}(\kappa)$ that satisfy
\small
\begin{align*}
        &\mathbb{P}\left\{\big\|f^*(x,z)-\hat{f}_n(x,z)\big\|^p > \kappa^p \right\} \leq K_{p}(\kappa,x) \exp{\Bigl(-n \beta_{p}(\kappa,x)\Bigr)},~\text{for a.e.}~x \in \mathcal{X},\ z \in \mathcal{Z},
        \\
        &\mathbb{P}\left\{\frac{1}{n}\sum_{k=1}^n \big\|f^*(x^k,z^k)-\hat{f}_n(x^k,z^k)\big\|^p > \kappa^p\right\} \leq \bar{K}_{p}(\kappa) \exp{\Bigl(-n \bar{\beta}_{p}(\kappa)\Bigr)}.
    \end{align*}
    \normalsize
\end{assumption}

   Both constants $K_{p}(\kappa,x)$ and $\beta_{p}(\kappa,x)$ in Assumption \ref{ass:reg constant} are independent of the decisions $z$. This assumption again differs from the literature as such constants normally depend on all covariates ($x$ and $z$). 
Theorem 2.2 of \cite{rigollet2023high} and Remark 12 of \cite{hsu2012random} provide some parametric regression setups that satisfy this assumption for $p=2$ with specific constants $K_{2}(\kappa,x,z),$ $\bar{K}_{2}(\kappa)$, $\beta_{2}(\kappa,x,z),$ and $\bar{\beta}_{2}(\kappa)$, and Lemma 10 of \cite{bertsimas2019predictions} provides some nonparametric regression setups that satisfy this assumption. For instance, OLS satisfies this assumption with decision-independent bounds $K_{2}(\kappa,x,z) = C_7 \exp(d_x + d_z),\ \bar{K}_{2}(\kappa) = C_8 \exp (d_x + d_z),\ \bar{\beta}_{2}(\kappa) = C_9 \frac{\kappa^2}{\sigma^2 d_y}$ and decision-dependent bound $\beta_{2}(\kappa,x,z) = C_{10} \frac{ \kappa^2}{\sigma^2 d_y (\|x\|^2 + \|z\|^2)}$ for some constants $C_7, C_8, C_9, C_{10} > 0$ \cite{rigollet2023high}. We can further identify an upper bound on $\|z\|^2$ (as $\mathcal{Z}$ is compact) to obtain a decision-independent constant $\beta_2(\kappa,x)$ for Assumption \ref{ass:reg constant}. 
 
 Using Assumption \ref{ass:reg constant}, we first derive the following proposition, which bounds the probability that the optimality gap---measured w.r.t.\ the true problem \eqref{eq: true problem}---of an optimal solution $\hat{z}_n^{D^3RO}(x)$ to \eqref{eq:main problem} exceeds $\kappa>0$. Using this result, we then establish the finite sample solution guarantee in Theorem \ref{thm:fssg}. 

    \begin{proposition}\label{prop:oos-guarantee}
	Suppose Assumptions \ref{ass:regression}, \ref{prop: concentration}, \ref{ass: lipschitz cont}, \ref{ass:risk}, and \ref{ass:reg constant} hold. Then for a.e.\ $x \in \mathcal{X}$ and any $\kappa >0$, there exist positive constants $\tilde{\Omega}(\kappa,x), \ \tilde{\omega}(\kappa,x)$ such that the solution of \eqref{eq:main problem} with the radius $\xi_n(\alpha,x)$ defined in \eqref{eq:radius def} and risk level $\alpha =\tilde{\Omega}(\kappa,x)\exp(-n\tilde{\omega}(\kappa,x))$ satisfies 
\begin{align}
\mathbb{P}\Bigl\{g(\hat{z}_n^{D^3RO}(x), x) - v^*(x) > \kappa \Bigr\} \leq 2\alpha.\label{eq:finite-sample-solution}
\end{align}
\end{proposition}
\begin{proof}\label{proof prop 5}
    Observe that
    \small
    \begin{align*}
        \mathbb{P} \big\{g(\hat{z}_n^{D^3RO}(x),x) > v^*(x) + \kappa\big\}
        =&\; \mathbb{P} \big\{g(\hat{z}_n^{D^3RO}(x),x) - \hat{v}_n^{D^3RO}(x) + \hat{v}_n^{D^3RO}(x) > v^*(x) + \kappa\big\} 
        \\
         \leq &\; \alpha + \mathbb{P} \big\{\hat{v}_n^{D^3RO}(x) > v^*(x) + \kappa\big\},
    \end{align*}
    \normalsize
    where the inequality follows from the probability inequality $\mathbb P\{A + B>a+b\} \le \mathbb P\{A>a\} + \mathbb P\{B>b\}$ with $a=0, b=v^*(x) + \kappa, A=g(\hat{z}_n^{D^3RO}(x),x) - \hat{v}_n^{D^3RO}(x), B=\hat{v}_n^{D^3RO}(x)$ and by using Theorem \ref{thm:fsg}. Suppose Assumption \ref{ass: lipschitz cont}(a) holds with $p \geq 1$ and Lipschitz constant $L_1(z)$, then following the proof of Lemma \ref{lemma: v inequality}, we have for any $z^*(x) \in S^*(x)$, $\mathbb{P} \{\hat{v}_n^{D^3RO}(x) > v^*(x) + 2L_1(z^*(x))\xi_n(\alpha,x)\} \leq \alpha$. If we choose the risk level $\alpha \in (0,1)$ s.t. $2L_1(z^*(x)) \kappa_{p,n}^{(1)}(\alpha,x) \leq \frac{\kappa}{2}$, and $2L_1(z^*(x)) \kappa_{p,n}^{(2)}(\alpha) \leq \frac{\kappa}{2}$, then we have $2L_1(z^*(x))\xi_n(\alpha,x) \leq \kappa$, which further implies
    \begin{align}
        \mathbb{P}\{g(\hat{z}_n^{D^3RO}(x),x) > v^*(x) + \kappa\} \leq 2\alpha.\label{eq:proposition-5}
    \end{align}
    Denoting $\bar{\kappa}:=\frac{\kappa}{8L_1(z^*(x))}$, by Assumption \ref{ass:reg constant}, if we choose $\alpha$ s.t. for a.e.\ $x \in \mathcal{X}$, 
    \begin{align*}
        \alpha \geq 4\max\left\{K_{p}(\bar{\kappa},x) \exp (-n \beta_{p} (\bar{\kappa},x)),\bar{K}_{p}(\bar{\kappa}) \exp (-n \bar{\beta}_{p} (\bar{\kappa}))\right\},
    \end{align*} 
    then the constant $\kappa_{p,n}^{(1)}(\alpha,x)$ in Eq.~\eqref{eq:radius def} satisfies $2L_1(z^*(x)) \kappa_{p,n}^{(1)} (\alpha,x) \leq \frac{\kappa}{2}$. Furthermore, by Eq.~\eqref{kappa2}, we know for a.e.\ $x \in \mathcal{X}$, if $\alpha \geq c_1(\exp(-c_2 n (\frac{\kappa}{4L_1(z^*(x))})^{1/s})$ with $s = \min\{p/d_y,1/2\}$ or $p/a$, then we have $2L_1(z^*(x)) \kappa_{p,n}^{(2)}(\alpha) \leq \frac{\kappa}{2}$. Therefore, there exist positive constants $\tilde{\Omega}_1(\kappa,x),\ \tilde{\omega}_1(\kappa,x)$ s.t. the solution of \eqref{eq:main problem} with risk level $\alpha = \tilde{\Omega}_1(\kappa,x) \exp{(-n \tilde{\omega}_1(\kappa,x))}$ satisfies Eq.~\eqref{eq:finite-sample-solution}.

    Suppose instead that Assumption \ref{ass: lipschitz cont}(b) holds with $p \geq 2$ and Lipschitz constant $L_2(z)$ with $\mathbb{E} [\| \nabla c(z,Y)\|^2] < +\infty$. Then following the proof of Lemma \ref{lemma: v inequality}, we have for any $z^*(x) \in S^*(x)$, 
    \begin{align*}
        \mathbb{P}\Bigl\{\hat{v}_n^{D^3RO}(x) > v^*(x) + (\mathbb{E}[\|\nabla c(z^*(x),Y)\|^2])^{1/2} \xi_n(\alpha,x) + 4L_2(z^*(x))\xi_n^2(\alpha,x)\Bigr\} \leq \alpha.
    \end{align*}
    Therefore, Eq.~\eqref{eq:proposition-5} is satisfied if we choose $\alpha \in (0,1)$ s.t.
    \begin{align}\label{eq: prop5 4b}
        \Bigl(\mathbb{E}[\|\nabla c(z^*(x),Y)\|^2]\Bigr)^{1/2} \xi_n(\alpha,x)
    + 4L_2(z^*(x))\xi_n^2(\alpha,x) \leq \kappa.
    \end{align}
    Using the similar analysis above, if we bound the smallest value of $\alpha$ using Assumption~\ref{ass:reg constant} and Eq.\ \eqref{eq:radius def}, then we can find positive constants $\tilde{\Omega}_2(\kappa,x),\ \tilde{\omega}_2(\kappa,x)$ that satisfy Eq.~\eqref{eq: prop5 4b}.
\end{proof}

\begin{theorem}[Finite sample solution guarantee]\label{thm:fssg}
 Suppose Assumptions \ref{ass:regression}, \ref{prop: concentration}, \ref{assum:lower-semicontinuous}, \ref{ass: lipschitz cont}, \ref{ass:risk}, and \ref{ass:reg constant} hold. Then for a.e.\ $x \in \mathcal{X}$ and any $\eta >0$, there exist positive constants $\Omega(\eta,x),\ \omega(\eta,x)$ such that the solution of \eqref{eq:main problem} under the ambiguity set $\hat{\mathcal{P}}_n(x,z;\xi_n(\alpha,x))$ defined in \eqref{eq:wass as} with the radius $\xi_n(\alpha,x)$ defined in \eqref{eq:radius def} and risk level defined as $\alpha ={\Omega}(\eta,x)\exp(-n{\omega}(\eta,x)) $ satisfies 
 \begin{align*}
     \mathbb{P}\{{\rm{dist}}(\hat{z}_n^{D^3RO}(x), S^*(x)) \geq \eta \} \leq 2\Omega(\eta,x)\exp{(-n \omega(\eta,x))}.
 \end{align*}

\end{theorem}
\begin{proof}
According to Proposition \ref{prop:oos-guarantee}, we have for all $\kappa >0$, there exist constants $\tilde{\Omega}(\kappa,x)>0$ and $\tilde{\omega}(\kappa,x)>0$ such that the solution of \eqref{eq:main problem}  with $\alpha = \tilde{\Omega}(\kappa,x)\exp(-n\tilde{\omega}(\kappa,x))$ satisfies Eq.~\eqref{eq:finite-sample-solution}. Suppose for some $\eta>0$ and $x\in\mathcal{X}$, we have $\text{dist} (\hat{z}_n^{D^3RO}(x),S^*(x)) \geq \eta$ and sample path. By Assumption \ref{assum:lower-semicontinuous}, $g(\cdot, x)$ is lower-semicontinuous on $\mathcal{Z}$ for a.e.\ $x\in\mathcal{X}$. Since $\mathcal{Z}$ is compact, according to Lemma 10 in \cite{kannan2025data}, we have $g(\hat{z}_n^{D^3RO}(x),x) > v^*(x) + \kappa(\eta,x)$ for some $\kappa(\eta,x)>0$ on that path. Therefore, for a.e.\ $x \in \mathcal{X}$
\begin{align*}
    & \mathbb{P}\{\text{dist}(\hat{z}_n^{D^3RO}(x),S^*(x)) \geq \eta\} \\
    \leq & \mathbb{P}\{g(\hat{z}_n^{D^3RO}(x),x) > v^*(x) + \kappa(\eta,x)\} 
    \leq  2\tilde{\Omega}(\kappa(\eta,x),x)\exp(-n\tilde{\omega}(\kappa(\eta,x),x)),
\end{align*}
where the desired result holds for constants $\Omega(\eta,x)=\tilde{\Omega}(\kappa(\eta,x),x)$ and $\omega(\eta,x)=\tilde{\omega}(\kappa(\eta,x),x).$
\end{proof}

\section{Decomposition Algorithm}
\label{sec:decomp}

In this section, we propose a specialized Bender's decomposition algorithm with nonlinear cuts to solve the resulting ER-${\rm D^3}$RO problem. For computational tractability, we consider a general two-stage ER-$\rm{D^3RO}$ problem with 1-Wasserstein distance and $\ell_1$ reference norm of the following form:
\begin{align}\label{eq: general 2 stage erd3ro}
    \min_{z\in\mathcal{Z}} \quad & c_z^\top z + \sup_{Q \in \hat{\mathcal{P}}_n(x,z)}\mathbb{E}_{Y \sim Q}\bigl[H(z,Y(x,z))\bigr],
\end{align}
where $z\in\mathcal{Z}$ is the first-stage decision, $c_z$ is the first-stage cost vector, and the second-stage optimal value function $H(z,Y(x,z))$ is defined as
\begin{align}\label{eq:rhs uncertainty}
H(z,Y(x,z)):= \min_{\omega} \{q ^\top \omega:   W \omega\ge T(z)Y(x,z) + h(z)\}.
\end{align}
In the above, $\omega \in \mathbb{R}^{d_{\omega}}$ denotes the second-stage decision with cost vector $q\in \mathbb{R}^{d_{\omega}}$, $Y(x,z) \in \mathbb{R}^{d_y}$ denotes the uncertainty that depends on the first-stage decision $z$ and covariate $x$, and both $T(z) \in \mathbb{R}^{M\times d_y}$ and $h(z) \in \mathbb{R}^{M}$ are affine in $z$. Here, we assume $Y(x,z) \in \mathbb{R}^{d_y}$, and thus the projection step can be omitted in this section. We consider the case in which the decision-dependent uncertainty $Y(x,z)$ appears on the right-hand side of the constraints in \eqref{eq:rhs uncertainty}. The case with objective uncertainty can be equivalently recast as \eqref{eq:rhs uncertainty} using an epigraph reformulation. Many real-world applications with decision-dependent uncertainty share the common form presented in Problem~\eqref{eq: general 2 stage erd3ro}--\eqref{eq:rhs uncertainty}. For example, in pricing problems where the uncertain demand depends on the pricing decision, the objective and/or constraints typically include a revenue term as the product of the first-stage pricing decision and the uncertain demand \citep{liu2022coupled,bertsimas2020predictive}. 

To derive the exact reformulation and develop a decomposition algorithm for Problem~\eqref{eq: general 2 stage erd3ro}, we first make the following two assumptions.
\begin{assumption}[Relatively complete recourse]
\label{ass:recourse}
    Second-stage problem \eqref{eq:rhs uncertainty} is feasible under every feasible first-stage decision $z \in \mathcal{Z}$ and every realization of $Y$.
\end{assumption}

\begin{assumption}[Sufficiently expensive recourse]
\label{ass:recourse2}
    Dual of the second-stage problem \eqref{eq:rhs uncertainty} is feasible under every feasible first-stage decision $z \in \mathcal{Z}$ and every realization of $Y$.
\end{assumption}

Note that Assumptions \ref{ass:recourse} and \ref{ass:recourse2} ensure that the second-stage problem \eqref{eq:rhs uncertainty} is feasible and bounded, i.e., $|H(z,Y(x,z))|<+\infty$, under every feasible first-stage $z \in \mathcal{Z}$ and every realization of $Y$. Together with $\mathcal{Z}$ nonempty and compact, this constitutes a special case of the general condition assumed earlier, $\mathbb{E}_\epsilon[|c(z,f^*(x,z) + \epsilon)|] <+\infty$. 
Next, we introduce an exact reformulation in Section~\ref{sec:reformulation} based on duality theory and present a specialized Bender's decomposition to solve the resulting model with convergence analysis in Section~\ref{sec:algorithm}.

\subsection{Exact Reformulation}
\label{sec:reformulation}

To begin, we first reformulate Problem~\eqref{eq: general 2 stage erd3ro}--\eqref{eq:rhs uncertainty} by changing the expectation from one taken w.r.t.\ $Y(x,z) \in \mathbb{R}^{d_y}$ to one taken w.r.t.\ $\epsilon \in \Xi \subseteq\mathbb R^{d_y}$.
Given a covariate $x$ and a first-stage decision $z$, observe that $\hat{f}_n(x,z)$ becomes a constant and we can translate the randomness in $Y$ into the randomness in the noise $\epsilon$. Define the empirical residual distribution of the noise term as $\hat L_n^{ER}:=\frac1n\sum_{k=1}^n \delta_{\hat\epsilon_n^{k}}$ 
and the corresponding ambiguity set on the residuals as $
\mathcal L_n(x):=\bigl\{L\in\mathcal P(\mathbb{R}^{d_y}): d_{W,1}\bigl(L,\hat L_n^{ER}\bigr)\le \xi_n(\alpha,x)\bigr\}$.
By the translation invariance of the 1-Wasserstein distance, we rewrite \eqref{eq: general 2 stage erd3ro} as
\begin{align}\label{eq: 2 stage epsilon}
    \min_{z\in\mathcal{Z}} \quad & c_z^\top z + \sup_{L \in \mathcal{L}_n(x)}\mathbb{E}_{\epsilon \sim L}\bigl[H(z,\hat{f}_n(x,z) + \epsilon)\bigr].
\end{align}
To present the reformulation of \eqref{eq: 2 stage epsilon} with 1-Wasserstein distance using the $\ell_1$ norm, we follow the proof of Theorem 6 in\ \cite{hanasusanto2018conic} and apply the necessary notation changes. 

\begin{proposition}
Suppose Assumption\ \ref{ass:recourse2} holds. Under the 1-Wasserstein ambiguity set with $\ell_1$ norm, Problem\ \eqref{eq: 2 stage epsilon} can be reformulated as follows:
\allowdisplaybreaks
\begin{subequations}\label{eq:general two stage problem reformulation}
\begin{align}
\min_{z,\lambda,\phi,\psi,\omega} \quad & c_z^\top z + \lambda \xi_n(\alpha,x) + \frac{1}{n} \sum_{k=1}^n q^\top \omega_k \nonumber &  \\
    \textrm{s.t.} 
    \quad & q^\top \phi_m \leq \lambda,\ \ \ q^\top \psi_m \leq \lambda & \forall m \in [d_y], \label{eq:wass dual constr 1} \\
    \quad & T(z)e_m \leq W\phi_m,\ \ \ -T(z)e_m \leq W\psi_m & \forall m \in [d_y], \label{eq:wass dual constr 2} \\
    \quad & z\in\mathcal{Z},\ \lambda \geq 0,\ \ \ \phi_m, \psi_m \in \mathbb{R}^{d_\omega} & \forall m \in [d_y],\label{eq:wass dual constr 3}\\
    \quad & W\omega_k\ge T(z)\bigl(\hat{f}_n(x,z) + \hat{\epsilon}_n^k\bigr) + h(z) & \forall k \in [n], \label{eq:second stage constraint} \\
    \quad & \omega_k \in \mathbb{R}^{d_{\omega}} & \forall k\in[n]. \label{eq:wass dual constr 4}
\end{align}
\end{subequations}
\end{proposition}
Note that the regression estimator $\hat f_n(x,z)$ in \eqref{eq:second stage constraint} can be constructed using a variety of regression methods, including parametric approaches (e.g., OLS and Lasso) and nonparametric approaches (e.g., kernel regression, CART, and neural networks). This makes Problem \eqref{eq:general two stage problem reformulation} highly nonconvex and computationally intractable to solve directly. Motivated by this challenge, we next develop a decomposition algorithm to solve Problem \eqref{eq:general two stage problem reformulation} more efficiently. We will test our decomposition algorithm with different parametric and nonparametric regression models $\hat f_n(x,z)$ in Section \ref{sec:experiments}.

\subsection{Decomposition Algorithm and Convergence Analysis}
\label{sec:algorithm}

In this section, we devise a specialized Bender's decomposition algorithm to solve Problem\ \eqref{eq:general two stage problem reformulation}. We treat $z,\lambda,\phi,\psi$ as first-stage decision variables subject to constraints \eqref{eq:wass dual constr 1}--\eqref{eq:wass dual constr 3}, and $\omega_k$ as second-stage decision variables subject to the second-stage constraints \eqref{eq:second stage constraint} and \eqref{eq:wass dual constr 4}. Then, using the second-stage optimal value function defined in \eqref{eq:rhs uncertainty} and an epigraph reformulation, we equivalently recast Problem \eqref{eq:general two stage problem reformulation}   as\vspace*{-0.05in} 
\begin{align} 
    \min_{z, \lambda, \phi, \psi, \mathbf{\Theta}} \quad & c_z^\top z + \lambda \xi_n(\alpha,x) + \frac{1}{n} \sum_{k=1}^n\Theta_k \label{eq:full benders problem} \\
    \textrm{s.t.}\quad & \text{Constraints}\ \eqref{eq:wass dual constr 1}-\eqref{eq:wass dual constr 3}, \nonumber \\
    \quad & \Theta_k \ge H(z,Y^k(x,z))\ \forall k\in[n], \nonumber
\end{align}
where $\mathbf{\Theta}:=\{\Theta_k\}_{k=1}^n$ and $Y^k(x,z):=\hat{f}_n(x,z) + \hat{\epsilon}_n^k$.

We start by decomposing Problem \eqref{eq:full benders problem} into a master problem and a set of subproblems. Suppose at iteration $t$, we have a first-stage solution $z^t$, which has been found by solving the master problem.  For each $k\in[n]$, we construct the following subproblem at $z^t$:
\begin{align} 
    \Omega_k(z^t) : = \max_{\pi \ge 0} \Bigl\{\pi^\top \Bigl(T(z^t) \bigl(\hat{f}_n(x,z^t) + \hat{\epsilon}_n^k\bigr) + h(z^t)\Bigr) :\ W^\top \pi = q\Bigr\}, 
    \tag{\textbf{SP}\textsuperscript{\it k}} 
    \label{eq:dual of second stage}
\end{align}
where we denote $\Pi = \{\pi \in \mathbb{R}_+^M: W^\top \pi = q\}$ as the feasible region of the subproblem. Note that \eqref{eq:dual of second stage} is the dual problem of $H(z^t,Y^k(x,z^t))$ defined in \eqref{eq:rhs uncertainty} with $Y^k(x,z^t)=\hat{f}_n(x,z^t) + \hat{\epsilon}_n^k$. Given the empirical residual $\hat{\epsilon}_n^k$, covariate $x$, and first-stage solution $z^t$, the term $\hat{f}_n(x,z^t) + \hat{\epsilon}_n^k$ becomes constant in the subproblem. We denote an optimal solution of \eqref{eq:dual of second stage} as $\pi^{k,t}$. These optimal dual solutions are then used to construct the following cuts
\begin{align}\label{eq:benders multi-cut}
\Theta_k \ \ge\ (\pi^{k,t})^\top\Bigl(T(z)(\hat{f}_n(x,z) + \hat{\epsilon}_n^k)+h(z)\Bigr)\ \ \ \forall k\in[n],
\end{align}
which are added to the master problem. The master problem at iteration $t$ is
\begin{align} 
        \min_{z, \lambda, \phi, \psi, \mathbf{\Theta}} \quad & c_z^\top z + \lambda \xi_n(\alpha,x) + \frac{1}{n} \sum_{k=1}^n\Theta_k \tag{{\bf MP}} \label{eq:master prob} \\
    \textrm{s.t.}\quad & \text{Constraints}\ \eqref{eq:wass dual constr 1}\ {-}\ \eqref{eq:wass dual constr 3}, \nonumber \\
    \quad & \Theta_k \ge (\pi^{k,\tau})^\top \Bigl(T(z) \bigl(\hat{f}_n(x,z) + \hat{\epsilon}_n^k\bigr) + h(z)\Bigr)\ \ \forall k\in[n],\ \tau \in [t-1], \nonumber
\end{align}
where the variables $\Theta_k$, along with the cuts present in \eqref{eq:master prob}, provide a lower approximation of the function $H(z,Y^k(x,z))$. Unlike the subproblems, for the master problem, while the empirical residual $\hat{\epsilon}_n^k$ is fixed, $\hat{f}_n(x,z)$ is no longer a constant. It is the learning model that is represented as a function of the decision variables $z$ in \eqref{eq:master prob}. Above, we present a multi-cut formulation of Bender's decomposition algorithm. Alternatively, one can use a single-cut formulation, in which the scenario-wise variables $\{\Theta_k\}_{k=1}^n$ are replaced by a single variable $\Theta$, and the corresponding cut is obtained by averaging the scenario-wise cuts over all scenarios.

Algorithm\ \ref{algo:erd3ro-benders} summarizes our specialized Bender's decomposition. It starts from the master problem \eqref{eq:master prob} with an empty cut set (or, with known bounds and cuts) by setting $t=1$. 
At each iteration $t$, we first solve the master problem \eqref{eq:master prob} and use its optimal objective value to obtain a lower bound. Then the optimal solution $z^t$ is passed to the subproblem \eqref{eq:dual of second stage} to generate one specialized Bender's cut \eqref{eq:benders multi-cut} for each scenario $k=1,\ldots,n$. These cuts are added to refine the master problem \eqref{eq:master prob}. The current incumbent solution can also be used to construct an upper bound $c_z^\top z^t + \lambda^t \xi_n(\alpha,x) + \frac{1}{n} \sum_{k=1}^n\Omega_k(z^t)$. The algorithm terminates when the current master problem's optimal solution satisfies the newly generated cuts.
Alternatively, one can also terminate the algorithm when the gap between the upper and lower bounds is within a pre-defined threshold.
\begin{algorithm}[ht!]
\caption{Specialized Bender's algorithm for solving ER-$\rm{D^3RO}$ \eqref{eq:general two stage problem reformulation}}
\label{algo:erd3ro-benders}
\small
\begin{algorithmic}[1]
\STATE \textbf{Initialize:} $\mathrm{LB}\gets -\infty$, $\mathrm{UB}\gets +\infty$, $t\gets 1$.
\WHILE{$\mathrm{UB}-\mathrm{LB}>\varepsilon$}

    \STATE \textbf{Solve\ \eqref{eq:master prob}:} Obtain the optimal solution $(z^{t},\lambda^{t},\phi^t,\psi^t,\{\Theta_k^{t}\}_{k=1}^n)$.
    \STATE \textbf{Update LB}: $\mathrm{LB}\gets$ optimal value of MP.

    \STATE \textbf{Solve\ \eqref{eq:dual of second stage} and obtain cut coefficients:}
    \FOR{$k=1$ \TO $n$}
        \STATE Evaluate $Y^k\bigl(x,z^{t}\bigr)\gets \hat f_n\bigl(x,z^{t}\bigr)+\hat\epsilon_n^{\,k}$ using the optimal solution $z^t$ from the\ \eqref{eq:master prob}.
        \STATE Solve\ \eqref{eq:dual of second stage} and obtain an optimal solution $\pi^{k,t}$.
    \ENDFOR

    \STATE \textbf{Update UB:} \\
    $   \mathrm{UB}\gets \min\left\{\mathrm{UB},\;
    c_z^\top z^{t} + \lambda^t \xi_n(\alpha,x)+\frac{1}{n}\sum_{k=1}^n (\pi^{k,t})^\top\bigl(T(z^{t})\,Y^k(x,z^{t})+h(z^{t})\bigr)
    \right\}.$
\IF{$\Theta_k^t \ge (\pi^{k,t})^\top\bigl(T(z^t)\bigl(\hat f_n(x,z^t)+\hat\epsilon_n^{\,k}\bigr)+h(z^t)\bigr)\ \forall k\in[n]$}
    \STATE \textbf{return} $(z^{t}, \lambda^{t},\phi^t,\psi^t, \{\Theta_k^{t}\}_{k=1}^n)$.
\ELSE
    \STATE Add the cuts $\Theta_k \;\ge\; (\pi^{k,t})^\top\bigl(T(z)\bigl(\hat f_n(x,z)+\hat\epsilon_n^{\,k}\bigr)+h(z)\bigr)\ \forall k\in[n]$ to\ \eqref{eq:master prob}.
\ENDIF

    \STATE $t\gets t+1$.
\ENDWHILE
\STATE \textbf{return} $(z^{t}, \lambda^{t},\phi^t,\psi^t, \{\Theta_k^{t}\}_{k=1}^n)$.
\end{algorithmic}
\normalsize
\end{algorithm}

Note that the cuts in \eqref{eq:benders multi-cut} are different from the traditional linear Bender's cuts, in which each iteration produces an affine inequality in the first-stage variable $z$. In contrast, the cuts in \eqref{eq:benders multi-cut} are generally nonconvex in the first-stage decision $z$ due to the term $T(z)(\hat{f}_n(x,z) + \hat{\epsilon}_n^k)$. As a result, incorporating these cuts makes the master problem \eqref{eq:master prob} highly nonconvex and may introduce additional computational challenges. Even so, we can still establish the finite convergence of our specialized Bender's algorithm. We first show that the recourse function $H(z,Y^k(x,z))$ is a piecewise nonlinear function in $z$ with a finite number of pieces in the next lemma.

\begin{lemma}\label{lemma:conv for nonlinear}
    Suppose Assumptions \ref{ass:recourse} and \ref{ass:recourse2} hold. The second-stage optimal value function $H(z,Y^k(x,z))$ is a piecewise nonlinear function in $z$ with a finite number of pieces for each $k\in[n]$.
\end{lemma}
\begin{proof}
Under both Assumptions \ref{ass:recourse} and \ref{ass:recourse2}, the second-stage recourse problem \eqref{eq:rhs uncertainty} is a feasible and bounded linear program given any first-stage decision $z$. Therefore, strong duality holds, i.e., $H(z,Y^k(x,z))=\Omega_k(z)$. The dual problem \eqref{eq:dual of second stage} is a maximization problem over the polyhedral feasible region $\Pi$. Since $\Pi$ does not depend on $z$ and has finitely many extreme points, $\Omega_k(z)$ can be expressed as the maximum of a finite collection of nonlinear functions in $z$, each associated with an extreme point of $\Pi$. Therefore, $\Omega_k(z)$ is a piecewise nonlinear function of $z$ with finitely many pieces.
\end{proof}

\begin{theorem}
Suppose Assumptions \ref{ass:recourse} and \ref{ass:recourse2} hold.  Assume there is a solution oracle that solves the master problem\ \eqref{eq:master prob} to optimality, and that the subproblems\ \eqref{eq:dual of second stage} are solved using a method that returns extreme points in finitely many iterations. Then the Bender's algorithm (Algorithm \ref{algo:erd3ro-benders}) converges to an optimal solution of Problem\ \eqref{eq:general two stage problem reformulation} in finitely many iterations.
\end{theorem}

\begin{proof}
Suppose $(z^*,\lambda^*,\phi^*,\psi^*,\{\Theta_k^{*}\}_{k=1}^n)$ is an optimal solution of Problem\ \eqref{eq:full benders problem}. Let $(z^{t},\lambda^{t},\phi^t,\psi^t,\{\Theta_k^{t}\}_{k=1}^n)$ denote an optimal solution of the master problem\ \eqref{eq:master prob} at iteration $t$. For each $k\in[n]$, denote $\pi^{k,t}$ to be an optimal solution of the subproblem\ \eqref{eq:dual of second stage} evaluated at $z=z^{t}$. 
Since each $\pi^{k,\tau}$ is a feasible solution of the dual problem\ \eqref{eq:dual of second stage} for any solution $z$, i.e., $\pi^{k,\tau}\in \Pi,\ \forall\tau\in[t]$, we have $\Theta_k^* \ge H(z^*,Y^k(x,z^*))=\max_{\pi\in\Pi} \pi^\top \bigl(T(z^*) \bigl(\hat{f}_n(x,z^*) + \hat{\epsilon}_n^k\bigr) + h(z^*)\bigr) \ge (\pi^{k,\tau})^\top \bigl(T(z^*) \bigl(\hat{f}_n(x,z^*) + \hat{\epsilon}_n^k\bigr) + h(z^*)\bigr)$ for all $k\in[n]$ and $\tau\in[t-1]$.
Therefore, $(z^*,\lambda^*,\phi^*,\psi^*,\Theta_k^*)$ is feasible for the current master problem \eqref{eq:master prob}, i.e., 
\begin{align}\label{eq:one-side}
c^\top z^*+\lambda^* \xi_n(\alpha,x) +\frac{1}{n}\sum_{k=1}^n\Theta_k^* \;\ge\; c^\top z^{t}+\lambda^{t} \xi_n(\alpha,x) +\frac{1}{n}\sum_{k=1}^n\Theta_k^{t}.
\end{align}

Thus, the master problem's optimal objective function value always provides a lower bound. We next suppose that the current master problem's optimal solution satisfies the newly generated cuts, i.e.,
\begin{align*}
\Theta_k^{t}\;\ge\;
(\pi^{k,t})^\top
\Big(
T(z^{t})(\hat f_n(x,z^{t})+\hat{\epsilon}_n^k)+h(z^{t})
\Big)
\;{=}\;
\Omega_k(z^{t})=H(z^t,Y^k(x,z^t))\ \forall k\in[n].
\end{align*}
Therefore, $(z^t,\lambda^t,\phi^t,\psi^t,\Theta_k^t)$ is feasible for Problem\ \eqref{eq:full benders problem}, i.e., 
\begin{align}\label{eq:other-side}
c^\top z^{t}+\lambda^{t} \xi_n(\alpha,x) +\frac{1}{n}\sum_{k=1}^n\Theta_k^{t} \;\ge\; c^\top z^*+\lambda^* \xi_n(\alpha,x) +\frac{1}{n}\sum_{k=1}^n\Theta_k^*.
\end{align}
Combining Eqs. \eqref{eq:one-side} and \eqref{eq:other-side}, we obtain $c^\top z^*+\lambda^* \xi_n(\alpha,x) +\frac{1}{n}\sum_{k=1}^n\Theta_k^* \;=\; c^\top z^{t}+\lambda^{t} \xi_n(\alpha,x) +\frac{1}{n}\sum_{k=1}^n\Theta_k^{t}$,
which implies that $(z^t,\lambda^t,\phi^t,\psi^t,\{\Theta_k^t\}_{k=1}^n)$ is an optimal solution of Problem\ \eqref{eq:full benders problem}. If instead the newly generated cut is violated at $(z^t,\lambda^t,\phi^t,\psi^t,\{\Theta_k^t\}_{k=1}^n)$, then after adding this cut, the point $(z^t,\lambda^t,\phi^t,\psi^t,\{\Theta_k^t\}_{k=1}^n)$ becomes infeasible for all subsequent master problems. As a result, the algorithm can never revisit the same master solution. 

Therefore, at each iteration, the algorithm either adds a new Bender's cut for each recourse function $H(z,Y^k(x,z))$ or terminates if the current solution satisfies the newly generated cuts. Since only finitely many distinct Bender's cuts can be generated by Lemma \ref{lemma:conv for nonlinear} and the same master solution cannot be revisited, the algorithm terminates finitely, and the returned solution is optimal for Problem\ \eqref{eq:general two stage problem reformulation}.
\end{proof}

\section{Computational Results}\label{sec:experiments}
In this section, we present the computational results of the specialized Bender's algorithm and compare the results across different regression models and optimization formulations under the 1-Wasserstein distance.
\subsection{Shipment Planning and Pricing Problem}
For our experiments, we consider a two-stage shipment planning and pricing problem, where we have $|\mathcal{I}|$ warehouses to satisfy the demand of a product at $|\mathcal{J}|$ customer sites, similar to the model in \cite{bertsimas2020predictive}. In the first stage, we aim to determine the price $z_1\in\mathbb{R}_+$ of this product and the production amounts $z_2\in\mathbb{R}_+^{|\mathcal{I}|}$ at $|\mathcal{I}|$ warehouses, with a unit production cost of $p_1\ge 0$. Then in the second stage, demand $Y_j\in\mathbb{R}_+$ is realized at each customer site $j\in\mathcal{J}$ and we must ship from the warehouses to satisfy all the demand. We ship $s_{ij}\ge 0$ units from warehouse $i$ to customer site $j$ at a unit transportation cost of $c_{ij}\ge 0$. We also have an option of last-minute production $t_i$ at warehouse $i$ with a higher unit production cost $p_2>p_1$. Specifically, we focus on the following model:\vspace*{-0.05in}
\begin{align}\label{eq:cvar two stage}
    \min_{z_1\in\mathbb{R}_+,z_2 \in\mathbb{R}_+^{|\mathcal{I}|}} p_1\sum_{i\in\mathcal{I}}z_{2,i}+ \mathbb{E}_Y [H(z,Y(x,z_1))]+ \rho \text{CVaR}_\theta (H(z,Y(x,z_1))),
\end{align}
where we aim to minimize the first-stage production cost and a combination of expectation and Conditional Value-at-Risk (CVaR) of the second-stage cost $H(z,Y(x,z_1))$. Here, the CVaR term can be reformulated as an expectation by the addition of an auxiliary decision variable \cite{rockafellar2000optimization}, resulting in a form considered in this paper. The parameter $\rho\ge 0$ is the weight of the CVaR term, and $\theta\in(0,1)$ specifies the risk aversion level, with $\text{CVaR}_\theta$ roughly averaging over the $100(1-\theta)\%$ worst-case outcomes.
Here, we assume that the random customer demand $Y(x,z_1)$ depends on some covariate information $x$ and our first-stage pricing decision $z_1$. Given a realization $Y^k(x,z_1)$, the second-stage problem is defined as follows:
\begin{subequations}\label{model:pricing}
\begin{align}
    H(z,Y^k(x,z_1)):=\min_{s_{ijk},t_{ik}\in \mathbb{R}_+} \quad & p_2\sum_{i\in\mathcal{I}}t_{ik}+\sum_{i \in \mathcal{I}} \sum_{j \in \mathcal{J}} c_{ij} s_{ijk} - z_1\sum_{j \in \mathcal{J}} Y^k_{j}(x,z_1)\\
    \textrm{s.t} \quad & \sum_{i \in \mathcal{I}} s_{ijk} \ge Y^k_{j}(x,z_1)\ \forall j \in \mathcal{J},\\
    \quad & \sum_{j \in \mathcal{J}} s_{ijk} \leq z_{2,i}+t_{ik}\ \forall i \in \mathcal{I},\vspace*{-0.05in}
\end{align}
\end{subequations}
where we aim to minimize the total production and transportation cost minus the revenue, subject to demand satisfaction and capacity constraints.

We randomly sample $|\mathcal{I}|=2$ warehouses and $|\mathcal{J}|=3$ customer sites on a $100\times 100$ grid. The unit transportation cost $c_{ij}$ is calculated based on the Euclidean distance between facility $i$ and customer site $j$, and the production costs are $p_1=5$ when done in advance and $p_2=100$ for last-minute production. Following the economics literature on the pricing-demand model (see \cite{petruzzi1999pricing,chen2004coordinating,chen2007risk}), we assume the ground truth demand model to be\vspace*{-0.05in} 
\begin{align}
    Y_j= \sum_{l\in\mathcal{L}}A_{jl}X_l^2-B_jz_1+\alpha_j+\epsilon_j\ \forall j \in \mathcal{J},\label{eq:ground-truth} \vspace*{-0.05in}
\end{align}
where we consider $|\mathcal{L}|=3$ demand-predictive covariates $X$, $z_1$ is our price, and $\epsilon_j\sim\mathcal{N}(0,1)$ is the zero-mean additive error term.  The coefficients $A_{jl}$ and $B_j$ are sampled uniformly between 3 and 5 for each $j\in\mathcal{J}$ and $l\in\mathcal{L}$, and the coefficient $\alpha_j$ is uniformly sampled between $1000$ and $2000$ for each $j\in\mathcal{J}$. Once these coefficients are fixed, we generate different replications of dataset $D_n:=\{(x^k,z^k,y^k)\}_{k=1}^n$ by sampling i.i.d. covariates $x^k$ from a Gamma distribution with shape parameter $k=2$ and scale parameter $\theta=3$ and i.i.d.\ price $z^k$ from a Beta distribution with parameters $\alpha =2$ and $\beta = 5$ (we then scale $z^k$ to be between 0 and 500). The error term $\epsilon_j^k$ is sampled from a Normal distribution with mean of 0 and standard deviation of 1. Then we obtain $y^k$ according to Eq.~\eqref{eq:ground-truth}.

\subsection{Experimental Setup}
Given a dataset $D_n=\{(x^k,z^k,y^k)\}_{k=1}^n$, we estimate the prediction function $\hat f_n(x,z_1)$ using three regression approaches: (i) parametric linear regression via OLS, (ii) nonparametric Nadaraya-Watson kernel regression using the uniform kernel, and (iii) nonparametric Rectified Linear Unit (ReLU) neural networks. Under the OLS approach, the estimated model is given by
\begin{align}\label{eq:estimate-model}
    \hat{f}_{n,j}(x,z_1) = \sum_{l \in \mathcal{L}} \hat{A}_{jl}x_l + \hat{B}_j z_1 + \hat{\alpha}_j\ \forall j \in \mathcal{J}. 
\end{align}
In this case, we have a misspecified model because $f^*\not\in\mathcal{F}$.
For the kernel regression, we take a weighted average of nearby data points, where the weight is determined by a uniform kernel function with a bandwidth parameter $h=8$. The resulting estimated model is given by
\begin{align}\label{eq:kernel regression}
\hat{f}_{n,j}(x,z_1) =
 \frac{\sum_{k=1}^{n} \mathbf{1}\left\{\lVert x-x^{k}\rVert^{2}+(z_1-z^{k})^{2}\le h^{2}\right\} \, y_j^{k}}
{\sum_{k=1}^{n} \mathbf{1}\left\{\lVert x-x^{k}\rVert^{2}+(z_1-z^{k})^{2}\le h^{2}\right\}}
\ \forall j \in \mathcal J.
\end{align}
To represent this kernel regression model in the downstream optimization, we further introduce binary variables and use the big-M technique to reformulate the indicator functions. 
We also adopt a neural network predictor with one hidden layer of 16 neurons and ReLU activation, given by
\begin{align}
&\boldsymbol{h}^{(0)} = [x,z_{1}],\ \boldsymbol{h}^{(1)} = \max\{\hat{\boldsymbol{v}}^{(1)} \boldsymbol{h}^{(0)} + \hat{\boldsymbol{\alpha}}^{(1)},0\} \nonumber\\
    &\hat{f}_{n,j}(x,z_1) = \hat{\boldsymbol{v}}^{(2)} \boldsymbol{h}^{(1)} + \hat{\boldsymbol{\alpha}}^{(2)}\ \  \forall j \in \mathcal{J}\label{eq:NN}
\end{align}
where $\hat{\boldsymbol{v}}^{(l)}$ and $\hat{\boldsymbol{\alpha}}^{(l)}$ for $l=1,2$ are the estimated weights and biases in the neural network.
Note that ReLU neural networks are mixed-integer representable \cite{anderson2020strong} and we directly embed the trained neural networks in our downstream optimization model using the Gurobi Machine Learning package \cite{GurobiML2026}. 

At a given new covariate $x$, the fitted regression models \eqref{eq:estimate-model}--\eqref{eq:NN} are then used to construct empirical residuals and ER-${\rm D^3}$RO models for computing in-sample solutions. Given an optimal in-sample solution $z_1^*$ and the new covariate $x$, we use the ground truth model \eqref{eq:ground-truth} to generate 1000 scenarios to evaluate the out-of-sample cost of the solution $z_1^*$. We generate 5 replications of datasets $D_n:=\{(x^k,z^k,y^k)\}_{k=1}^n$ and report the average of out-of-sample costs over these 5 independent runs.

Numerical tests are conducted on the Ohio Supercomputer Center, on a Linux system running Red Hat Enterprise Linux 9.4 and equipped with an Intel(R) Xeon(R) Gold 6148 CPU @ 2.40GHz. Since Model \eqref{eq:cvar two stage} and its DRO reformulations involve non-convex terms, we use Gurobi 12.0.0 coded in Python 3.11.0 for solving all non-convex programming models (with NonConvex parameter set to 2), with a computational time limit set to three hours.

\subsection{Numerical Results of Specialized Bender's Algorithm}
We compare the computational performance of our specialized Bender's algorithm with that of solving the problem directly using Gurobi. Since the multi-cut formulation consistently outperforms the single-cut formulation in our problem setting, we report the results of the multi-cut Bender's algorithm. Table\ \ref{tab:benders comparison} summarizes the in-sample cost (cost), optimality gap (gap), and runtime (time) of the two approaches across three different regression approaches: OLS (O), uniform kernel regression (K), and ReLU neural network (N) under different sample sizes.  For these experiments, if a method does not finish within the time limit (3 hours), we report its best obtained in-sample cost and the optimality gap at the end of the time limit. 
\begin{table}[ht!]
\centering
\caption{Performance comparison between the Bender's algorithm and the direct solution by Gurobi}
\label{tab:benders comparison}
\scriptsize
\setlength{\tabcolsep}{4pt}
\renewcommand{\arraystretch}{1.1}
\resizebox{\textwidth}{!}{
\begin{tabular}{c|c|cc|cc|cc|cc}
\hline
& & \multicolumn{2}{c|}{$n=1000$} & \multicolumn{2}{c|}{$n=2000$} & \multicolumn{2}{c|}{$n=3000$} &
\multicolumn{2}{c}{$n=4000$}\\
& & Gurobi & Bender's& Gurobi & Bender's & Gurobi & Bender's & Gurobi & Bender's \\
\hline
& Cost & \bf{-186292} & -186280 & -129456 & \bf{-129457} & -187615 & \bf{-215392} &  -143694 & \bf{-270851}\\
O & Gap & \bf{0.003\%} & 0.03011\% & \bf{0.00242\%} & 0.00597\% & 219.94\% & \textbf{0.01\%} & 607.63\% & 7.20\%\\
& Time & \textbf{2106} & 6450 & 6730 & \textbf{3616} & 10800 & \textbf{5488} & 10800 & \textbf{10647}\\
\hline
& & \multicolumn{2}{c|}{$n=100$} & \multicolumn{2}{c|}{$n=200$} & \multicolumn{2}{c|}{$n=300$} &
\multicolumn{2}{c}{$n=400$}\\
& & Gurobi & Bender's & Gurobi & Bender's & Gurobi & Bender's & Gurobi & Bender's \\
\hline
& Cost & \bf{-3367905} & -3367893 
& \bf{-3053800} & -3053705 
& \bf{-3478859} & -3478841 &\bf{-3244664} &-3244652 \\
K & Gap & inf & \textbf{0.0034\%}
& 0\% & 0\%
& inf & \textbf{0.0006\%} & inf & 0.0008\%\\
& Time & 2181 & \textbf{35}
& 954 & \textbf{451}
& 4574 & \textbf{1806} & 7077 & \bf{4684}
\\
& \# of inf gap & 1/5 & \textbf{0/5}
& 0/5 & 0/5
& 1/5 & \textbf{0/5} & 3/5 & \bf{0/5}
\\
\hline
& Cost & -2371330 & \textbf{-2371342} 
& 5752162 & \textbf{-2111285} 
& 727948 & \textbf{-2612521} &1219342&\textbf{-2105438} \\
 & Gap & \textbf{0.003\%} & 0.004\%
& inf& \textbf{0.004\%}
& inf & \textbf{0.006\%} & inf & \textbf{0.0029\%}\\
N& Time & 206 & \textbf{38}
& 2608 & \textbf{127}
& 7080 & \textbf{75}& 8483 & \textbf{116}\\
& \# of inf gap & 0/5 & 0/5
& 1/5 & \textbf{0/5}
& 3/5 & \textbf{0/5} & 2/5 & \textbf{0/5}\\
& Unbounded & 1/5 & 1/5
& 0/5 & 0/5
& 4/5 & 4/5 & 4/5 & 4/5\\
\hline
\end{tabular}
}
\end{table}

Table\ \ref{tab:benders comparison} shows that our Bender's algorithm scales much better than the Gurobi solver in two ways.
First, it achieves shorter runtime in almost all instances, with the only exception occurring in the OLS case when $n=1000$. Second, our algorithm yields significantly smaller optimality gaps compared to the direct solution by Gurobi. In particular, under the kernel and neural network regression, Gurobi frequently returns infinite optimality gaps (reported as ``\# of inf gap'' in Table\ \ref{tab:benders comparison}) across different sample sizes as it cannot obtain a valid lower bound within the time limit, whereas our algorithm consistently achieves small optimality gaps. Notice that, in some problem instances, neural network predictions fail to capture the relationship between demand and the pricing decision accurately. As a result, the model may drive the price to infinity while still predicting positive demand, which leads to an unbounded optimization problem (reported as ``Unbounded'' in Table\ \ref{tab:benders comparison}). For the bounded instances, Gurobi still fails to produce a valid lower bound in most cases, whereas our algorithm can find a near-optimal solution within a few hundred seconds. These results show a clear computational advantage and stronger numerical stability of our proposed algorithm across regression approaches and sample sizes.

\vspace*{-0.01in}
\subsection{Comparison of Different Models and Methods}
We compare the proposed model \eqref{eq:main problem} with its non-DRO (but decision-dependent) counterpart \eqref{eq: ersaa}, and its decision-independent (but DRO) counterpart denoted by ER-DRO. 
For the decision-independent counterpart, we fit a regression model $\hat{f}_n(x)$ without the $z$ variable. To have a fair comparison, in this section, we use sample sizes $n=100,\ldots,400$ for all regression models, including OLS. Also, for each sample size and each replication within that sample size, all three regression models see the same observations. For both ER-${\rm D^3}$RO and ER-DRO models, we use leave-one-out cross-validation to select the best radius $\xi$, which is chosen from the candidate set $\{1,10,50,100\}$. We apply Bender's algorithm to solve all three models and report the average out-of-sample cost over five independent runs in Figure\ \ref{fig:pricing-decision-dependency}.
\begin{figure}[ht!]
    \centering

\pgfplotsset{
    erdrowok/.style={black, solid, mark=*},
    erddsaao/.style={black, dashed, mark=square*},
    erddsaak/.style={black, dashed, mark=triangle*},
    erdthreeo/.style={black, solid, mark=diamond*},
    erdthreek/.style={black, solid, mark=o}
}
    \begin{minipage}[t]{0.32\textwidth}
        \centering
        \resizebox{\textwidth}{!}{%
        \begin{tikzpicture}
          \begin{axis}[
            xlabel={Sample Size $n$},
            ylabel={Average Out-of-Sample Cost},
            xtick={100,200,300,400},
            legend to name=commonlegend,
            legend columns=3,
            legend style={
                font=\scriptsize,
                text=black,
                draw=none,
                /tikz/every even column/.append style={column sep=0.2cm}
            }
          ]
            \addplot[erdthreeo] coordinates {
              (100,-1056016.909)
              (200,-1144750.66)
              (300,-1533715.495)
              (400,-1397359.333)
            };
            \addplot[erdthreek] coordinates {
              (100,-2538215.099)
              (200,-2320965.481)
              (300,-2708484.258)
              (400,-2458363.703)
            };
            \addplot[erdrowok] coordinates {
              (100,0)
              (200,0)
              (300,0)
              (400,0)
            };
            \addplot[erddsaao] coordinates {
              (100,-920676.1494)
              (200,-996716.6597)
              (300,-1418437.984)
              (400,-1044973.463)
            };
            \addplot[erddsaak] coordinates {
              (100,-2538054.822)
              (200,-2217397.424)
              (300,-2686444.389)
              (400,-2376120.377)
            };
            \legend{
                ER-${\rm D^3}$RO-O,
                ER-${\rm D^3}$RO-K,
                ER-DRO-O/K,
                ER-DD-SAA-O,
                ER-DD-SAA-K
            }
          \end{axis}
        \end{tikzpicture}%
        }

        \vspace{1mm}
        {\footnotesize (a) Different regression models}
    \end{minipage}
    \hfill
    \begin{minipage}[t]{0.32\textwidth}
        \centering
        \resizebox{\textwidth}{!}{%
        \begin{tikzpicture}
          \begin{axis}[
            xlabel={Sample Size $n$},
            ylabel={Average Out-of-Sample Cost},
            xtick={100,200,300,400},
            scaled y ticks=true
          ]
            \addplot[erddsaao] coordinates {
              (100,-920676.1494)
              (200,-996716.6597)
              (300,-1418437.984)
              (400,-1044973.463)
            };
            \addplot[erdthreeo] coordinates {
              (100,-1056016.909)
              (200,-1144750.66)
              (300,-1533715.495)
              (400,-1397359.333)
            };
          \end{axis}
        \end{tikzpicture}%
        }

        \vspace{1mm}
        {\footnotesize (b) OLS regression}
    \end{minipage}
    \hfill
    \begin{minipage}[t]{0.32\textwidth}
        \centering
        \resizebox{\textwidth}{!}{%
        \begin{tikzpicture}
          \begin{axis}[
            xlabel={Sample Size $n$},
            ylabel={Average Out-of-Sample Cost},
            xtick={100,200,300,400},
            scaled y ticks=true
          ]
            \addplot[erddsaak] coordinates {
              (100,-2538054.822)
              (200,-2217397.424)
              (300,-2686444.389)
              (400,-2376120.377)
            };
            \addplot[erdthreek] coordinates {
              (100,-2538215.099)
              (200,-2320965.481)
              (300,-2708484.258)
              (400,-2458363.703)
            };
          \end{axis}
        \end{tikzpicture}%
        }

        \vspace{1mm}
        {\footnotesize (c) Kernel regression}
    \end{minipage}

    \vspace{1mm}
    \centering
    \ref{commonlegend}
\vspace*{-0.13in}
    \caption{Out-of-sample cost comparison between ER-${\rm D^3}$RO, ER-DRO, and ER-DD-SAA models with different sample sizes $n$ under different regression models.}
    \label{fig:pricing-decision-dependency}
\end{figure}
\paragraph{Comparison of Different Regression Models}
We first compare the impact of using different regression models in Figure\ \ref{fig:pricing-decision-dependency}(a). Because the neural network regression results in many unbounded optimization problems due to poor prediction quality, as reported in Table\ \ref{tab:benders comparison}, its out-of-sample cost is set to 0 for these unbounded cases. This results in high out-of-sample costs. Therefore, we do not include it in the out-of-sample cost comparison.
Figure\ \ref{fig:pricing-decision-dependency}(a) shows that for both ER-DD-SAA and ER-$\rm{D^3RO}$, the uniform kernel regression (ER-DD-SAA-K and ER-${\rm D^3}$RO-K) consistently achieves lower out-of-sample cost compared to the OLS regression (ER-DD-SAA-O and ER-${\rm D^3}$RO-O) across all sample sizes, which illustrates the value of nonparametric regression. 

\paragraph{Value of Decision Dependency}
We now compare the ER-${\rm D^3}$RO with its decision-independent counterpart ER-DRO in Figure \ref{fig:pricing-decision-dependency}. 
When the demand $Y_j$ is independent of decision $z_1$, Problem \eqref{eq:cvar two stage} becomes unbounded and drives $z_1$ to infinity at optimality. As a result, the out-of-sample scenarios of the demand generated by the ground truth model \eqref{eq:ground-truth} become all zero and thus lead to a cost of 0 (as reported in ER-DRO-O/K). On the other hand, our ER-${\rm D^3}$RO model returns much smaller out-of-sample costs. This result highlights the benefit of accounting for decision dependence.

\paragraph{Value of Distributional Robustness}
We now compare the ER-DD-SAA model with the ER-${\rm D^3}$RO model in Figure\ \ref{fig:pricing-decision-dependency}. Figures \ref{fig:pricing-decision-dependency}(b) and \ref{fig:pricing-decision-dependency}(c) are close-ups of Figure \ref{fig:pricing-decision-dependency}(a) under OLS and kernel regression, respectively. We observe that the ER-$\rm{D^3RO}$ models with decision-dependent ambiguity sets achieve lower out-of-sample costs under both kernel regression and OLS compared to ER-DD-SAA. This shows that the ER-${\rm D^3}$RO model is less sensitive to errors in the empirical distribution and is more robust compared to the ER-DD-SAA model.

\section{Conclusion and Future Work}\label{sec: future work}
In this paper, we considered a contextual stochastic program where the uncertainty could be affected by both covariate information and our decisions. We introduced an empirical residuals framework, where the uncertainty on the prediction is considered in a distributionally robust manner with the Wasserstein distance-based ambiguity set. Within this framework, we established several statistical guarantees for the proposed ER-$\rm{D^3RO}$ model, including asymptotic optimality, rate of convergence, finite-sample certificate guarantees, and finite-sample solution guarantees. To solve the resulting ER-$\rm{D^3RO}$ model more efficiently, we further developed a specialized Bender's decomposition algorithm with nonconvex cuts, which can converge to the optimal solution in finitely many iterations under mild conditions. We tested our model and algorithm on a shipment planning and pricing problem with linear regression via OLS, kernel regression, and neural network regression. The numerical results indicated that the proposed Bender's algorithm can improve computational efficiency over Gurobi in most of the sample sizes and regression models. In addition, the ER-$\rm{D^3RO}$ model outperforms both the decision-dependent ER-SAA model and its decision-independent counterpart, which shows the benefit of incorporating distributional ambiguity and the importance of accounting for decision dependence.
One possible future direction is to design data augmentation and cross-validation algorithms that can achieve good empirical and theoretical performance in a limited data regime. Investigating appropriate forms of decision-dependent radii backed with theoretical guarantees also merits further research. 

\vspace*{-0.05in}
\bibliographystyle{siamplain}
\bibliography{references}

\end{document}